\documentclass[12pt,leqno]{amsart}
\usepackage[latin9]{inputenc}
\usepackage{verbatim,ifpdf}
\usepackage[usenames,dvipsnames]{xcolor}
\usepackage{enumitem}
\setlist{leftmargin=*}
\setlist[enumerate]{label=\textup{(\arabic*)}}
\usepackage{amsthm}
\usepackage{amstext}
\usepackage{amssymb}

\numberwithin{equation}{section}
\numberwithin{figure}{section}

\usepackage{amsthm}
\usepackage{amscd}
\usepackage{bm}
\ifpdf
\usepackage[pdftex,pdfstartview=FitH,pdfborderstyle={/S/B/W 1}]{hyperref}
\else
\usepackage[hypertex]{hyperref}
 \fi
\theoremstyle{plain}
\newtheorem{theorem}{Theorem}[section]
\newtheorem{lemma}[theorem]{Lemma}
\newtheorem{conjecture}[theorem]{Conjecture}
\newtheorem{corollary}[theorem]{Corollary}
\newtheorem{proposition}[theorem]{Proposition}
\newtheorem{claim}[theorem]{Claim}

\theoremstyle{definition}
\newtheorem{definition}[theorem]{Definition}
\newtheorem{remark}[theorem]{Remark}
\newtheorem{example}[theorem]{Example}

\newcommand{\mdim}{\mathrm{mdim}}
\newcommand{\diam}{\mathrm{diam}}
\newcommand{\widim}{\mathrm{Widim}}
\newcommand{\dist}{\mathrm{dist}}
\newcommand{\supp}{\mathrm{supp}}
\newcommand{\norm}[1]{\left|\!\left|#1\right|\!\right|}
\newcommand{\vol}{\mathrm{vol}}
\newcommand{\ocap}{\mathrm{ocap}}
\newcommand{\mesh}{\mathrm{mesh}}
\newcommand{\mmdim}{\mathrm{mdim}_{\mathrm{M}}}
\newcommand{\Z}{\mathbb{Z}}
\newcommand{\R}{\mathbb{R}}

\newcommand{\safenopunct}{\texorpdfstring{\nopunct}{}}

\makeatother

\begin{document}

\title[Mean dimension of $\mathbb{Z}^k$-actions]{Mean dimension of $\mathbb{Z}^k$-actions}

\author{Yonatan Gutman, Elon Lindenstrauss, Masaki Tsukamoto}



\date{\today}

\begin{abstract}
Mean dimension is a topological invariant for dynamical systems that is meaningful for systems with infinite dimension and infinite entropy.
Given a $\mathbb{Z}^k$-action on a compact metric space $X$, we study the
following three problems closely related to mean dimension.
\begin{enumerate}
  \item When is $X$ isomorphic to the inverse limit of finite entropy systems?
  \item Suppose the topological entropy $h_{\mathrm{top}}(X)$ is infinite.
       How much topological entropy can be detected if one considers $X$ only up to a given level of accuracy? How fast does this amount of entropy grow as the level of resolution becomes finer and finer?
  \item When can we embed $X$ into the $\mathbb{Z}^k$-shift on the infinite dimensional cube
        $([0,1]^D)^{\mathbb{Z}^k}$?
\end{enumerate}
These were investigated for $\mathbb{Z}$-actions in [Lindenstrauss, Mean dimension, small entropy factors and an embedding theorem,
Inst. Hautes \'{E}tudes Sci. Publ. Math. \textbf{89} (1999) 227-262],
but the generalization to $\mathbb{Z}^k$ remained  an open problem.
When $X$ has the marker property, in particular when $X$ has a completely aperiodic minimal factor, we completely solve (1) and a natural interpretation of (2), and give a reasonably satisfactory
answer to (3).

A key ingredient is a new method to continuously partition every orbit into good pieces.
\end{abstract}
\thanks{Y.G was partially supported by the Marie Curie grant PCIG12-GA-2012-334564
and by the National Science Center (Poland) grant 2013/08/A/ST1/00275. E.L. awknowledges the support of ERC AdG Grant 267259.
M.T. was supported by Grant-in-Aid for Young Scientists (B)
25870334 from JSPS}
\maketitle

\section{Introduction} \label{section: introduction}

\subsection{Main results} \label{subsection: main results}

Mean dimension is a topological invariant of dynamical systems introduced by
Gromov \cite{Gromov}.
Just like topological entropy measures the number of bits per second to describe a point in a system,
mean dimension measures the number of parameters per second.
A basic example is the shift action on the Hilbert cube $[0,1]^{\mathbb{Z}}$, whose
mean dimension is 1.
This system has infinite dimension and infinite topological entropy;
mean dimension, however, provides a useful numerical invariant for such large dynamical systems.
Soon after the introduction of mean dimension, some basic properties of mean dimension, in particular its relation to toplogical entropy and to embedding questions were studied by Benjamin Weiss and the second named author (\cite{Lindenstrauss--Weiss, Lindenstrauss}).

Both \cite{Lindenstrauss--Weiss, Lindenstrauss} focus on the case of $\Z$-actions. However, there is a substantial difference between these two papers in this respect:
The paper \cite{Lindenstrauss--Weiss} studied the basic theory of mean dimension, and the restriction to $\Z$-action in this paper was purely because of expositional reasons. As explained in that paper,
all its main results can be generalized to the
actions of discrete amenable groups without any essential change given a result due to Ornstein and Weiss on subadditive functions on amenable groups (cf.\ Lemma~\ref{lemma: Ornstein--Weiss} below) that is given in \cite[Appendix]{Lindenstrauss--Weiss} (this lemma is implicit in Ornstein--Weiss \cite[Ch.~I, Sec.~2 and 3]{Ornstein--Weiss} and written explicitly for similar purposes also in Gromov \cite[p. 336]{Gromov}).
Recently Hanfeng Li \cite{Li} successfully generalized theses results even to the much larger class of sofic groups.

The paper \cite{Lindenstrauss} studied more delicate questions and in that paper specific properties of $\mathbb{Z}$ were used. How to generalize these results to $\Z ^ k$-actions was one of the questions left open in \cite{Lindenstrauss--Weiss, Lindenstrauss}, and it is precisely this question which we are finally able to address in this paper.

Our motivation to develop the generalization is two-fold.
The first is a purely theory motivated: There is a tradition in ergodic theory and dynamical systems to consider actions of more general group. A notable example in this vein is the paper \cite{Ornstein--Weiss}, where much of the body of knowledge of ergodic theory of $\Z$-actions is generalized to the context of actions of amenable groups.

Moreover, some of the most natural and interesting examples of systems with nontrivial mean dimension arise in the context of $\Z ^ k$ actions (or $\R ^ k$-actions, which for our point of view are almost equivalent).
Indeed, the concept of mean dimension  was introduced by Gromov \cite{Gromov} in order to study
dynamical systems in geometric analysis from the viewpoint of mean dimension.
In most of the systems considered in \cite{Gromov}, the acting groups are more complicated than $\mathbb{Z}$.
For example, \cite[Chapter 4]{Gromov} studied a dynamical system consisting of complex subvarieties in $\mathbb{C}^n$.
In this case $\mathbb{C}^{n}$ and its lattice $\mathbb{Z}^{2n}$ are the acting groups, the action being by translation.
Readers can find many more examples in \cite[Chapters 3 and 4]{Gromov}.

Fix a positive integer $k$. Let $X$ be a compact metric space with a continuous action $T:\mathbb{Z}^k\times X\to X$.
We call $(X,\mathbb{Z}^k,T)$ a dynamical system.
We often abbreviate $(X,\mathbb{Z}^k,T)$ to $X$ and
denote $T(n,x)$ by $T^nx$ for $n\in \mathbb{Z}^k$ and $x\in X$,
The mean dimension of the system $X$ is denoted by $\mdim(X)$.
Its definition is given in Section \ref{section: review of mean dimension}.
The paper studies three problems closely related to $\mdim(X)$.

\subsection*{The first problem:\safenopunct}
\textit{When can we approximate a dynamical system $X$ arbitrarily well by finite topological entropy systems?}
More precisely,  when
is $X$ isomorphic to the inverse limit of $\mathbb{Z}^k$-actions
\[ \dots\rightarrow X_n \rightarrow X_{n-1}\rightarrow \dots\rightarrow X_2\rightarrow X_1 \]
such that every $X_n$ has a finite topological entropy?

The inverse limit of finite entropy systems is always \textit{zero mean dimensional}
(\cite[Proposition 6.11]{Lindenstrauss}).
So $\mdim(X)=0$ is a necessary condition. We conjecture the following

\begin{conjecture} Let $\Gamma$ be a discrete amenable group (in particular, consider $\Gamma = \Z ^ k$).
A $\Gamma$-system $(X, \Gamma, T)$ is an inverse limit of zero entropy systems iff it has zero mean dimension.
\end{conjecture}

Our first main result is a partial result towards this conjecture for $\Gamma = \Z ^ k$. The following definition is a variant of \cite[Def.~2]{Downarowicz-minimal-models}

\begin{definition}\label{def: marker}
Let $(X, \Gamma,T)$ be a dynamical system.
It is said to have the \textbf{marker property} if for any finite subset $F \subset \Gamma$ there exists an open set
$U\subset X$ such that
\[
X = \bigcup_{n\in \Gamma} T^n U
\]
and that $U\cap T^nU=\emptyset$ for all non-identity $n \in F $.
\end{definition}

In particular it follows from the definition that if $(X, \Gamma, T)$ has the marker property the $\Gamma$-action is free, i.e. $T^nx\neq x$ for all $x \in X$ and $n \in \Gamma$. To conform with the standard terminology for $\Z$-actions, we shall say a dynamical system $(X, \Gamma, T)$ is aperiodic if the underlying $\Gamma$-action is free.
Clearly, the marker property is preserved by extensions: if $(X, \Gamma, T)$ is a $\Gamma$-system and $(Y, \Gamma, S)$ a factor with the marker property then so does $X$. A large class of systems with this property is the class of \textbf{aperiodic minimal systems}: a system $(X, \Gamma, T)$ is said to be \textbf{minimal} if every $\Gamma$-orbit is dense; it is an \textbf{aperiodic minimal system} if the system is in addition aperiodic.

It is not clear whether any aperiodic $\Z^k$-system has the marker property or not. For further discussion of this question see \cite{Gutman 2, Gutman 3}

\begin{theorem}[Cf.~{\cite[Prop.~6.14]{Lindenstrauss}}] \label{thm: zero mean dimension and topological entropy}
Suppose a $\Z^k$-system $X$ has the marker property.
Then $\mdim(X)=0$ if and only if $X$ is isomorphic to the inverse limit of finite topological entropy systems.
\end{theorem}

The result cited, \cite[Prop.~6.14]{Lindenstrauss}, gives the theorem for actions of $\Z$; even though it is assumed there that $X$ has an aperiodic minimal factor, what is actually used in the proof is that $X$ has the marker property.
Note that if a minimal system is not aperiodic then its topological entropy is zero, hence
Theorem \ref{thm: zero mean dimension and topological entropy} completely answers the first problem
for minimal $\Z ^ k$-systems.

\medskip

\subsection*{The second problem:\safenopunct}
\textit{Suppose the topological entropy of $X$ is infinite. How much topological entropy can be detected if one considers $X$ only up to a given level of accuracy? How fast does this amount of entropy grow as the level of resolution becomes finer and finer?}

Suppose $X$ is endowed with a distance $d$.
For $\varepsilon>0$ we define $S(X,\varepsilon,d)$ as the amount of entropy that can be detected in $X$ by considering it only up to accuracy $\varepsilon$ with respect to the metric $d$ --- see Section \ref{section: metric mean dimension is equal to mean dimension}
for the precise definition.

The topological entropy $h_{\mathrm{top}}(X)$ is the limit of $S(X,\varepsilon,d)$ as $\varepsilon$ goes to zero; the limit, unlike the $S(X,\varepsilon,d)$, does not depend on the choice of $d$.
When $h_{\mathrm{top}}(X)=\infty$, we are interested in its growth.
This motivates the introduction of the \textbf{metric mean dimension} $\mmdim(X,d)$:
\[ \mmdim(X,d) = \liminf_{\varepsilon\to 0}\frac{S(X,\varepsilon,d)}{|\log\varepsilon|}.\]

Weiss and the second named author (\cite[Theorem 4.2]{Lindenstrauss--Weiss}) proved $\mmdim(X,d)\geq \mdim(X)$
for any distance $d$.
Our second main result is the following.
\begin{theorem}[Cf.~{\cite[Thm.~4.3]{Lindenstrauss}}] \label{thm: metric mean dimension is equal to mean dimension}
Suppose $(X, \Z^k, T)$ is a $\Z^k$-system with the marker property.
Then there exists a distance $d$ on $X$, compatible with the topology, satisfying
\begin{equation}\label{eq:metric mean dim}
 \mmdim(X,d) = \mdim(X).
\end{equation}
\end{theorem}

Therefore, at least for extensions of completely aperiodic minimal systems, the growth of $S(X,\varepsilon,d)$ can be
controlled by the mean dimension $\mdim(X)$.

Again we conjecture that \eqref{eq:metric mean dim} holds for any $\Z^k$-system $X$, without having to assume the marker property, and indeed not just for $\Z^k$-systems but for any countable amenable group.

\subsection*{The third problem:\safenopunct}
Let $D$ be a positive integer, and consider the $\mathbb{Z}^k$-shift on
$([0,1]^D)^{\mathbb{Z}^k}$:
\[ \mathbb{Z}^k\times ([0,1]^D)^{\mathbb{Z}^k} \to ([0,1]^D)^{\mathbb{Z}^k}, \quad
   (a, (x_n)_{n\in \mathbb{Z}^k}) \to (x_{n+a})_{n\in \mathbb{Z}^k}.\]
\textit{Given a dynamical system $X$, when can we equivariantly embed $X$ into $([0,1]^D)^{\mathbb{Z}^k}$?}
Beboutov showed that any action of $\R$ on a compact metric space whose fixed points can be embedded in an interval can be embedded in the space of continuous function on $\R$, with the natural action of $\R$ (cf.~\cite{Kakutani-proof-Beboutov}).
In his Ph.D. thesis, Jaworski \cite{Jaworski} proved that if a system $(X,\mathbb{Z},T)$ is finite dimensional and has no periodic points then
there exists an embedding from $X$ into the $\mathbb{Z}$-shift $[0,1]^{\mathbb{Z}}$. It was left unclear whether the periodic points are the only obstruction for such an embedding.
Weiss and the second named author \cite{Lindenstrauss--Weiss} observed that mean dimension
is another obstruction.
The mean dimension of $\mathbb{Z}^k$-shift $([0,1]^D)^{\mathbb{Z}^k} $ is $D$. Hence
if $X$ is embedded into it, its mean dimension $\mdim(X)$ is not greater than $D$.
The paper \cite[Proposition 3.5]{Lindenstrauss--Weiss} constructed an aperiodic minimal system for $\Z$ whose
mean dimension is strictly greater than $1$; this construction was generalized to arbitrary amenable groups in \cite{Krieger-minimal-large}.
These constructions give examples of
aperiodic minimal systems which cannot be embedded into $[0,1]^{\mathbb{Z}^k}$.

Even if the $\Z^k$-action is free, mean dimension is not the only obstruction for embedding. In \cite{Lindenstrauss--Tsukamoto} a $\Z$-minimal system is given with $\mdim (X)=D/2$ that cannot be embedded in $([0,1]^{D})^{\mathbb{Z}}$ for any $D \in \mathbb{N}$, and it is straightforward to adapt this example to $\Z^d$ (and even to general discrete amenable groups).

Our third main theorem is a partial converse to this result.

\begin{theorem}[Cf.~{\cite[Thm.~5.1]{Lindenstrauss}}] \label{thm: embedding theorem}
Suppose $(X, \Z^k, T)$ is a $\Z^k$-system with the marker property.
If it satisfies
\[ \mdim(X) < \frac{D}{2^{k+1}},\]
then there exists an embedding from $X$ into $([0,1]^{2D})^{\mathbb{Z}^k}$.
\end{theorem}

In \cite[Thm.~5.1]{Lindenstrauss} it was shown that a $\Z$-systems with the marker property\footnote{Formally the assumption was that $X$ has a non-periodic minimal factor.} with $\mdim(X) < {D}/{36}$ can be embedded in $([0,1]^{D})^{\mathbb{Z}}$.
For technical reasons, it is useful for us to have as a target the $\mathbb{Z}^k$-shift on $([0,1]^{2D})^{\mathbb{Z}^k}$ rather than $([0,1]^D)^{\mathbb{Z}^k}$.

Note also that in our condition $\mdim(X)<D/2^{k+1}$, the constant involved is likely far from optimal. Presumably for an aperiodic $\Z^k$-system if $\mdim(X) < \frac{D}{2}$ then $X$ can be embedded in  $([0,1]^D)^{\mathbb{Z}^k}$;
for $\mathbb{Z}$-actions, it is conjectured in \cite{Lindenstrauss--Tsukamoto}
that if a system $(X,\mathbb{Z},T)$ satisfies
\[ \mdim(X)<\tfrac{D}2, \quad  \tfrac1n{\dim \left(\{x|T^n x=x\}\right)} <\tfrac{D}{2}  \quad (\forall n\geq 1)\]
then there is an embedding from $X$ into $([0,1]^D)^{\mathbb{Z}}$.

We can obtain an embedding result with the optimal constant of $1/2$ if we make a stronger assumption about $X$ than in Theorem~\ref{thm: embedding theorem} above.
Any closed invarinat subspace of $\{1,2,\dots,l\}^{\mathbb{Z}^k}$ on which $\Z^k$ acts freely (whether minimal or not) can be seen to satisfy the marker propert. We shall call such a system an \textbf{aperiodic symbolic system}. Following \cite{Gutman--Tsukamoto} (who treated the case of $\Z$-actions) we show that if $X$ is assumed to have an aperiodic symbolic factor one can get an optimal embedding constant:

\begin{theorem}[Cf.~{\cite[Corollary 1.8]{Gutman--Tsukamoto}}]\label{thm: embedding for extensions of symbolic systems}
Suppose $X$ has an aperiodic symbolic factor.
If $\mdim(X)<D/2$ then there exists an embedding from $X$ into $([0,1]^D)^{\mathbb{Z}^k}$.
\end{theorem}

We note that the proof of Theorem~\ref{thm: embedding for extensions of symbolic systems} is substantially less complicated than that of Theorem~\ref{thm: embedding theorem}.

\subsection{Main ideas} \label{subsection: main ideas}

To explain the main ideas, we briefly review the proof of Jaworski's theorem
\cite{Jaworski}:
If a system $(X,\mathbb{Z},T)$ is finite dimensional and has no periodic points, then
it can be embedded into $[0,1]^\mathbb{Z}$.
We follow the presentation of Auslander \cite[Chapter 13]{Auslander}.
For a continuous map $f:X\to [0,1]$ we define $I_f:X\to [0,1]^\mathbb{Z}$ by
$I_f(x) := (f(T^nx))_{n\in \mathbb{Z}}$.
This is always $\mathbb{Z}$-equivariant.
By the Baire category theorem, it is enough to prove that for any distinct $x_1,x_2\in X$
we can find closed neighborhoods $A_i$ of $x_i$ such that
\begin{equation} \label{eq: Jaworski theorem}
   \{f\in C(X,[0,1])|\, I_f(A_1)\cap I_f(A_2) = \emptyset\}
\end{equation}
is open dense in the space $C(X,[0,1])$ of continuous functions $f:X\to [0,1]$.

Fix a natural number $N > 2\dim X$.
Since $X$ has no periodic points, we can find $j_1<j_2<\dots<j_N$ such that
$2N$ points
\[ T^{j_1}x_1, \dots, T^{j_N}x_1, T^{j_1}x_2, \dots, T^{j_N}x_2\]
are all distinct.
(If $x_2=T^mx_1$ for some $m$ then we take $0,m+1, 2(m+1), \dots, (N-1)(m+1)$. Otherwise we take $0,1,2,\dots,N-1$.)
There exist open neighborhoods $U_i$ of $x_i$ such that
$2N$ sets $T^{j_n}U_i$ $(1\leq n\leq N, i=1,2)$ are pairwisely disjoint.
We take closed neighborhoods $A_i\subset U_i$ of $x_i$.
Let $\varphi:X\to [0,1]$ be a cut-off function such that $\varphi=1$ on the union of the $2N$ sets $T^{j_n}A_i$ and
its support is contained in the union of $T^{j_n}U_i$.

Obviously the set (\ref{eq: Jaworski theorem}) is open.
Take arbitrary $f\in C(X,[0,1])$ and $\delta>0$.
The condition $N>2\dim X$ implies that \textit{generic} continuous maps from $X$ to $[0,1]^N$ are
\textit{topological embeddings}.
(See \cite[Thm.~V2]{Hurewicz--Wallman}; it can be also deduced from
Lemmas \ref{lemma: approximation by linear map} and \ref{lemma: embedding of simplicial complex} below.)
Hence we can find
an embedding $F:X\to [0,1]^N$ such that
the $\ell^\infty$-distance between $F(x)$ and $(f(T^{j_n} x))_{1\leq n\leq N}$ is less than $\delta$ for all $x\in X$.

We define a perturbation $g$ of $f$ as follows:
If $x\in T^{j_n}U_i$ for some $n$ and $i$ then
\[ g(x) = (1-\varphi(x))f(x) + \varphi(x)F(T^{-j_n}x)_n.\]
Otherwise set $g(x) =f(x)$.
This satisfies $|g(x)-f(x)|<\delta$ and for $x\in A_1\cup A_2$
\[ (g(T^{j_1}x),\dots,g(T^{j_N}x)) = F(x).\]
Then $g$ satisfies $I_g(A_1)\cap I_g(A_2)=\emptyset$ because $F$ is an embedding.
This shows the density of (\ref{eq: Jaworski theorem}) and finishes the proof.

The above proof has three important steps:
\begin{enumerate}
  \item Find good pieces $T^{j_1}x_1,\dots, T^{j_N}x_1$ and $T^{j_1}x_2,\dots,T^{j_N}x_2$ of the orbits of $x_1$ and $x_2$.
  \item Find an embedding $F$ approximating $I_f|_{\{j_1,\dots,j_N\}}$ by using $N>2\dim X$.
  \item Define a perturbation $g$ of $f$ by ``painting $F$ on the good pieces of orbits''.
\end{enumerate}
The proofs of our main theorems develop similar three steps
for infinite dimensional systems.
We do not need a substantial change in step (3).
In step (2) we replace ``embedding'' with ``$\varepsilon$-embedding'', which is an approximative version
of embedding (see Section~\ref{section: review of mean dimension}.)
The condition $N>2\dim X$ is replaced with a condition on mean dimension.
So we crucially need mean dimension theory in step (2).
But machineries for this step were already developed for $\mathbb{Z}$-actions
in \cite{Lindenstrauss}; the $\mathbb{Z}^k$-case does not require a new idea.

The central issue is step (1).
\textit{We need to continuously partition every orbit into good pieces
where steps (2) and (3) work well}.
For dealing with this problem the paper \cite{Lindenstrauss} introduced a new topological analogue of
the Rokhlin tower lemma in ergodic theory.
But this does not work for $\mathbb{Z}^k$.
The failure of the tower lemma technique is the main barrier to the generalization to $\mathbb{Z}^k$.

The first idea to overcome this difficulty is the use of \textbf{Voronoi tiling}.
Gutman \cite{Gutman 1} is the first paper using Voronoi tiling in mean dimension theory.
We develop this technique further.
Suppose $\mathbb{Z}^k$ continuously acts on a compact metric space $X$.
Take a small open set $U\subset X$.
For each point $x\in X$ we consider the set
\[ C(x) = \{n\in \mathbb{Z}^k|\, T^nx\in U\}, \]
and let
\[ \mathbb{R}^k = \bigcup_{n\in C(x)} V(x,n), \quad
   V(x,n) = \{u\in \mathbb{R}^k|\, \forall m\in C(x): |u-n|\leq |u-m|\}, \]
be the Voronoi tiling associated with $C(x)$.
We try to use an appropriate piece $V(x,n)$ (or the lattice points $V(x,n)\cap \mathbb{Z}^k$)
as a substitute for $\{j_1,\dots,j_N\}$ in the proof of Jaworski's theorem.
This idea perfectly works if $X$ has an aperiodic symbolic factor.
We prove Theorem \ref{thm: embedding for extensions of symbolic systems} by this method.
There is some analogy here to \cite{Lightwood 1, Lightwood 2} where Voronoi tilings were used to obtain a
$\mathbb{Z}^k$-analogue of the Krieger embedding theorem
\cite{Krieger} for symbolic subshifts.

In general the tiles $V(x,n)$ do not depend continuously on $x\in X$, and
the above idea cannot be directly applied to Theorems \ref{thm: zero mean dimension and topological entropy},
\ref{thm: metric mean dimension is equal to mean dimension} or
\ref{thm: embedding theorem}.
So we introduce the second idea: \textbf{Adding one dimension}.
This is the most important new idea in this paper.
We take a cut-off function $\phi:U\to [0,1]$ supported in the above open set $U$, and consider the set
\[  \{(n,1/\phi(T^nx))|\, n\in \mathbb{Z}^k: \phi(T^nx)\neq 0\}. \]
Note that this is a subset of $\mathbb{R}^{k+1}$. So \textit{we go up one dimension higher}.
Let $\mathbb{R}^{k+1} = \bigcup_{n\in \mathbb{Z}^k} V(x,n)$ be the associated Voronoi decomposition.
We take a large number $H$ and set $W(x,n) = V(x,n)\cap (\mathbb{R}^k\times \{-H\})$.
Then we get the decomposition
\[ \mathbb{R}^k\times \{-H\} = \bigcup_{n\in \mathbb{Z}^k} W(x,n).\]
This \textit{does} depend continuously on $x\in X$.
Thus we can use $W(x,n)$ as a substitute for $\{j_1,\dots,j_N\}$ in Jaworski's theorem.
This establishes step (1).

\subsection{Open problems} \label{subsection: open problems}

The following three questions are the main open problems arising from the paper.
\begin{itemize}
 \item \textit{Can one remove the marker property assumption in Theorems~\ref{thm: zero mean dimension and topological entropy}
 and~\ref{thm: metric mean dimension is equal to mean dimension}?}
 We conjecture that $\mdim(X)=0$ is equivalent to the condition that $X$ is isomorphic to the inverse limit of
 finite entropy systems without any additional condition.
 We also conjecture that for any system $X$ there always exists a distance $d$ satisfying $\mdim(X)=\mmdim(X,d)$.
 These conjectures are open even for $\mathbb{Z}$-actions.

 \item \textit{What is the optimal condition to ensure the existence of the embedding into the $\mathbb{Z}^k$-shift
 $([0,1]^D)^{\mathbb{Z}^k}$?}
 Probably it is enough to assume $\mdim(X)<D/2$ and some conditions on the periodic points.

 \item \textit{How can one extend the theory to actions of non-commutative groups?}
 Our new technique in this paper uses the Euclidean geometry in an essential way, and hence
 cannot be generalized to other groups directly.
\end{itemize}

\subsection{Organization of the paper} \label{subsection: organization of the paper}

In Section \ref{section: review of mean dimension} we recall some basic definitions related to mean dimension.
In Section \ref{section: proof of embedding and symbolic system} we prove
the embedding theorem for extensions of aperiodic symbolic systems (Theorem
\ref{thm: embedding for extensions of symbolic systems}) by using Voronoi tiling.
Although this motivates more difficult arguments in later sections, it is logically independent of other theorems.
In Section \ref{section: adding one dimension} we explain the technique of adding one dimension.
In Section \ref{section: small entropy factors}
we review the idea of small boundary property (a dynamical analogue of totally disconnectedness)
and solve the problem of approximating a zero mean dimensional system by finite entropy ones
(Theorem \ref{thm: zero mean dimension and topological entropy}).
In Section \ref{section: metric mean dimension is equal to mean dimension}
we prove the existence of
a distance $d$ satisfying $\mdim_M(X,d)=\mdim(X)$ (Theorem
\ref{thm: metric mean dimension is equal to mean dimension}).
In Section \ref{section: embedding problem}
we prove the embedding theorem for extensions of aperiodic minimal systems (Theorem \ref{thm: embedding theorem}).
Sections \ref{section: small entropy factors}, \ref{section: metric mean dimension is equal to mean dimension}
and \ref{section: embedding problem} are almost independent of each other;
Section \ref{section: embedding problem} is substantially more complicated
than other sections.

\section{Review of mean dimension} \label{section: review of mean dimension}

Here we recall some basic facts on mean dimension.
For the details, see Gromov \cite{Gromov} and Lindenstrauss--Weiss \cite{Lindenstrauss--Weiss}.
Throughout the paper we assume that $k$ is a fixed positive integer.
For a positive integer $N$ we set
\[ [N] =\{0,1,2,\dots,N-1\}^k \subset \mathbb{Z}^k.\]
All simplicial complexes are implicitly assumed to be finite, i.e. consist of only finitely many simplices.

Let $(X,d)$ be a compact metric space.
Let $Y$ be a topological space, and $f:X\to Y$ a continuous map.
For a positive number $\varepsilon$ the map $f$ is called an \textbf{$\varepsilon$-embedding}
if $\diam f^{-1}(y) < \varepsilon$ for all $y\in Y$.
We define $\widim_\varepsilon(X,d)$ as the minimum integer $n\geq 0$ such that
there exist an $n$-dimensional simplicial complex $P$ and an $\varepsilon$-embedding
$f:X\to P$.

Let $P$ be a simplicial complex, and $V$ a Banach space.
A map $f:P\to V$ is said to be linear if
it has the following form on every face $\Delta\subset P$:
\[ f\left(\sum_{i=0}^n \lambda_iv_i\right) = \sum_{i=0}^n \lambda_i f(v_i) \]
where $v_i$ are the vertices of $\Delta$ and $\lambda_i$ are nonnegative numbers satisfying $\sum_{i=0}^n \lambda_i =1$.

\begin{lemma} \label{lemma: approximation by linear map}
Let $C\subset V$ be a convex subset,
$(X,d)$ a compact metric space,
and $f:X\to C$ a continuous map.
Suppose $\varepsilon$ and $\delta$ are positive numbers satisfying
\[ d(x,y)<\varepsilon \Longrightarrow \norm{f(x)-f(y)} < \delta.\]
Let $\pi:X\to P$ be an $\varepsilon$-embedding from $(X,d)$ to a simplicial complex $P$.
Then, after replacing $P$ by a sufficiently finer subdivision, there exists a
linear map $g:P\to C$ satisfying
\[ \norm{f(x)-g(\pi(x))} < \delta, \quad \forall x\in X.\]
\end{lemma}
\begin{proof}
By subdividing $P$, we can assume $\diam \,\pi^{-1}(O(v)) < \varepsilon$ for all vertices $v$ of $P$.
Here $O(v)$ is the open star of $v$, namely the union of the relative interiors of faces containing $v$.
For each vertex $v\in P$ we define $g(v)$ as follows:
If $\pi^{-1}(O(v))\neq \emptyset$ then we choose a point $x_v\in \pi^{-1}(O(v))$ and set $g(v) = f(x_v)$.
If $\pi^{-1}(O(v))= \emptyset$, then we choose $g(v)\in C$ arbitrarily.
We define a linear map $g:P\to C$ by extending it linearly on every face.

Take $x\in X$, and let $v_0,\dots,v_n$ be the vertices of the face of $P$ containing $\pi(x)$ in its relative interior.
Suppose $\pi(x) = \sum_{i}\lambda_i v_i$ with $0<\lambda_i<1$ and $\sum \lambda_i=1$.
Since $\pi(x) \in O(v_i)$ for each $i$, we get $d(x,x_{v_i})<\varepsilon$.
Hence $\norm{f(x)-f(x_{v_i})} < \delta$. So
\[ \norm{f(x)-g(\pi(x))} = \norm{\sum \lambda_i(f(x)-f(x_{v_i}))} < \delta.\]
\end{proof}

Suppose $(X,d)$ is a compact metric space with a continuous action $T:\mathbb{Z}^k\times X\to X$.
For a subset $\Omega\subset \mathbb{Z}^k$ we define  a new distance $d_\Omega$ on $X$ by
\[ d_\Omega(x,y) = \sup_{n\in \Omega}d(T^nx,T^ny).\]
We define the \textbf{mean dimension} $\mdim(X,\mathbb{Z}^k,T)$ (often abbreviated to $\mdim(X)$) by
\[ \mdim(X,\mathbb{Z}^k,T) = \lim_{\varepsilon\to 0}
   \left(\lim_{N\to \infty}\frac{1}{N^k} \widim_\varepsilon(X,d_{[N]})\right).\]

The limit with respect to $N$ in the above equation exists because of the following sub-additivity and invariance:
\begin{equation*}
   \begin{split}
   \widim_\varepsilon(X,d_{\Omega_1\cup\Omega_2}) &\leq
   \widim_\varepsilon(X,d_{\Omega_1}) + \widim_\varepsilon(X,d_{\Omega_2}), \quad (\Omega_1,\Omega_2\subset \mathbb{Z}^k),\\
   \widim_\varepsilon(X,d_{a+\Omega}) &= \widim_\varepsilon(X,d_\Omega), \quad
   (a\in \mathbb{Z}^k, \Omega\subset \mathbb{Z}^k),
   \end{split}
\end{equation*}
where $a+\Omega$ denotes the set $\{a+n|\, n\in \Omega\}$.
From this property and the standard division argument
\[ \lim_{N\to \infty}\frac{1}{N^k} \widim_\varepsilon(X,d_{[N]})
   = \inf_{N\geq 1}\frac{1}{N^k} \widim_\varepsilon(X,d_{[N]}).\]
Note that the value of $\mdim(X)$ is independent of the choice of a distance $d$ compatible with the topology, hence mean dimension is a topological invariant.

\begin{example}
Let $X=([0,1]^D)^{\mathbb{Z}^k}$ with the standard shift action of $\mathbb{Z}^k$.
Then its mean dimension is $D$.
\end{example}

In Section \ref{section: proof of embedding and symbolic system} we use more general F{\o}lner sequences.
For a subset $\Omega \subset \mathbb{Z}^k$ and $R>0$ we define
$\partial^{\mathbb{Z}}_R\Omega$ as the set of $n\in \mathbb{Z}^k$ such that there exist $m\in \Omega$ and
$m'\in \mathbb{Z}^k\setminus \Omega$ satisfying $|n-m|\leq R$ and $|n-m'|\leq R$.
We set $\mathrm{int}^{\mathbb{Z}}_R\Omega = \Omega\setminus \partial^{\mathbb{Z}}_R\Omega$.
Here the symbol $\mathbb{Z}$ indicates that these are defined for subsets of $\mathbb{Z}^k$.
In later sections we will consider $\partial$ and $\mathrm{int}$ for subsets of $\mathbb{R}^k$.
A sequence $\{\Omega_n\}_{n\geq1}$ of finite subsets of $\mathbb{Z}^k$ is called a \textbf{F{\o}lner sequence}
if for every $R>0$
\[ \lim_{n\to \infty} \frac{|\partial^{\mathbb{Z}}_R\Omega_n|}{|\Omega_n|} =0.\]
For example the sequence $\{[n]\}_{n\geq 1}$ is F{\o}lner.
We use the following lemma
(Gromov \cite[p. 336]{Gromov} and Lindenstrauss--Weiss \cite[Appendix]{Lindenstrauss--Weiss}).
This lemma is originally due to Ornstein--Weiss \cite[Chapter I, Sections 2 and 3]{Ornstein--Weiss}
and holds for general amenable groups.
\begin{lemma} \label{lemma: Ornstein--Weiss}
Let $h: \{\text{finite subsets of }\,\mathbb{Z}^k\}\to \mathbb{R}$ be a nonnegative function satisfying
the following three conditions.
\begin{itemize}
  \item   If $\Omega_1\subset \Omega_2$, then $h(\Omega_1)\leq h(\Omega_2)$.
  \item  $h(\Omega_1\cup \Omega_2) \leq h(\Omega_1)+h(\Omega_2)$.
  \item   For any $a\in \mathbb{Z}^k$ and any finite subset $\Omega \subset \mathbb{Z}^k$, we have
      $h(a+\Omega)=h(\Omega)$.
\end{itemize}
Then for any F{\o}lner sequence $\{\Omega_n\}_{n\geq 1}$ in $\mathbb{Z}^k$, the limit of the sequence
\[ \frac{h(\Omega_n)}{|\Omega_n|} \quad (n\geq 1) \]
exists and is independent of the choice of a F{\o}lner sequence.
\end{lemma}

For example the function $h(\Omega) = \widim_\varepsilon(X,d_{\Omega})$ (defined for a dynamical system $(X,\mathbb{Z}^k,T)$)
satisfies the above three conditions.
So for any F{\o}lner sequence $\{\Omega_n\}_{n\geq 1}$ in $\mathbb{Z}^k$ we have
\begin{equation} \label{eq: mean dimension by Folner sequence}
 \mdim(X) = \lim_{\varepsilon\to 0}\left(\lim_{n\to \infty}\frac{\widim_\varepsilon(X,d_{\Omega_n})}{|\Omega_n|}\right).
\end{equation}

\section{An embedding theorem: proof of Theorem \ref{thm: embedding for extensions of symbolic systems}}
\label{section: proof of embedding and symbolic system}

In this section we prove Theorem \ref{thm: embedding for extensions of symbolic systems}.
Fix a positive integer $D$.
Let $(Z,\mathbb{Z}^k,S)$ be an aperiodic \textit{zero dimensional} system.
Here $\dim Z=0$ means that clopen (closed and open) subsets form an open basis.
Symbolic systems are zero dimensional.
Let $\pi:X\to Z$ be an extension. Namely $(X,\mathbb{Z}^k,T)$ is a dynamical system with a $\mathbb{Z}^k$-equivariant
continuous surjection $\pi$.
Let $C(X,[0,1]^D)$ be the space of continuous maps $f:X\to [0,1]^D$ with the uniform norm topology.
For $f:X\to [0,1]^D$ we define the map $I_f:X\to ([0,1]^D)^{\mathbb{Z}^k}$ by $I_f(x)=(f(T^nx))_{n\in \mathbb{Z}^k}$.

\begin{theorem} \label{thm: embedding of zero dimensional extension}
If $\mdim(X)<D/2$ then for a dense $G_\delta$ subset of $f\in C(X,[0,1]^D)$ the map
\[ (I_f,\pi):X\to ([0,1]^D)^{\mathbb{Z}^k}\times Z, \quad x\mapsto (I_f(x),\pi(x))\]
is an embedding.
\end{theorem}
Theorem \ref{thm: embedding for extensions of symbolic systems} in the introduction
follows from this theorem.
\begin{proof}[Proof of Theorem \ref{thm: embedding for extensions of symbolic systems}]
Suppose $Z$ is an aperiodic subsystem of \[\{1,2,\dots,l\}^{\mathbb{Z}^k}\] and $X$ is its extension with $\mdim(X)<D/2$.
We can \textit{topologically} embed the \textit{space} $[0,1]^D\times \{1,2,\dots,l\}$ into $[0,1]^D$.
Hence we can \textit{dynamically} embed
the \textit{system} $([0,1]^D)^{\mathbb{Z}^k}\times \{1,2,\dots,l\}^{\mathbb{Z}^k}$ into
$([0,1]^D)^{\mathbb{Z}^k}$. So we can also embed $([0,1]^D)^{\mathbb{Z}^k}\times Z$ into $([0,1]^D)^{\mathbb{Z}^k}$.
Then Theorem \ref{thm: embedding of zero dimensional extension} implies that
we can embed $X$ into the system $([0,1]^D)^{\mathbb{Z}^k}$ .
\end{proof}

In the rest of this section we always assume that $Z$ is an aperiodic zero dimensional system
and $\pi:X\to Z$ is an extension with $\mdim(X)<D/2$.
We set $K=[0,1]^D$ for simplicity of the notation.
Let $d$ be a distance on $X$.
For each $\delta>0$ the set
\[ \left\{f\in C(X,K)\middle|  \text{\parbox{2.9in}{\centering $(I_f,\pi):X\to K^{\mathbb{Z}^k}\times Z$ is a $\delta$-embedding with respect to $d$}}\right\}\]
is open in $C(X,K)$.
Consider the $G_\delta$ subset
\[ \bigcap_{n\geq 1}
  \left\{f\in C(X,K)\middle| \text{\parbox{2.9in}{\centering $(I_f,\pi):X\to K^{\mathbb{Z}^k}\times Z$ is a $\tfrac1n$-embedding with respect to $d$}}\right\}.\]
This is equal to the set of $f\in C(X,K)$ such that $(I_f,\pi)$ is an embedding.
Therefore Theorem \ref{thm: embedding of zero dimensional extension} follows from the next proposition.
\begin{proposition} \label{prop: main proposition for embedding of zero dimensional extension}
Let $f:X\to K$ be a continuous map.
For any $\delta>0$ there exists a continuous map $g:X\to K$ satisfying
\begin{enumerate}
\item $|f(x)-g(x)|<\delta$ for all $x\in X$,
\item $(I_g,\pi):X\to K^{\mathbb{Z}^k}\times Z$ is a $\delta$-embedding with respect to the distance~$d$.
\end{enumerate}
\end{proposition}

We recall the following well-known fact about simplicial complexes:
\begin{lemma} \label{lemma: embedding of simplicial complex}
Let $n$ be a positive integer, and $P$ a simplicial complex with $\dim P < n/2$.
Then almost every linear map $\varphi:P\to [0,1]^n$ is a topological embedding.
\end{lemma}
\begin{proof}
Let $v_1,\dots,v_s$ be the vertices of $P$.
The space of linear maps $f:P\to [0,1]^n$ is identified with
\[ \left([0,1]^n\right)^{\{v_1,\dots,v_s\}} \]
which is endowed with the standard Lebesgue measure.
The above ``almost every'' is defined with respect to this measure.
For almost every choice of vectors $u_1,\dots,u_s\in [0,1]^n$,
all $(n+1)$-tuples $u_{i_1},\dots,u_{i_{n+1}}$ $(i_1<\dots<i_{n+1})$ are affinely independent.
Then for such a choice of $u_1,\dots,u_s$ the linear map $\varphi:P\to [0,1]^n$ defined by $\varphi(v_i)=u_i$
becomes a topological embedding because $2\dim P+2 \leq n+1$.
\end{proof}

The following lemma established in particular that an aperiodic zero-dimensional system has the marker property (cf.~Definition \ref{def: marker}). Note however that the set $U$ in the lemma below is not just an open set as in Definition~\ref{def: marker} but clopen, which greatly simplifies the subsequent constructions.

\begin{lemma} \label{lemma: clopen marker}
For any $L>0$ there exists a clopen subset $U\subset Z$ such that
\begin{equation} \label{eq: clopen marker}
    Z=\bigcup_{|n|<L}S^n U
\end{equation}
and $U\cap S^nU=\emptyset$ for all non-zero $n\in \mathbb{Z}^k$ with $|n|<L$.
\end{lemma}
\begin{proof}
This lemma is close to the argument of Lightwood \cite[Section 4]{Lightwood 1}.
From $\dim Z=0$, there exists a clopen covering $\{V_1,\dots,V_a\}$ of $Z$ such that
\[ V_i\cap S^nV_i = \emptyset \quad (\forall 1\leq i\leq a, 0<|n|<L).\]
We inductively define clopen sets $U_1,\dots,U_a$ by $U_1=V_1$ and
\[ U_{i+1} = U_i\cup \left(V_{i+1}\setminus \bigcup_{|n|<L} S^nU_i\right).\]
Then $U_i\cap S^nU_i=\emptyset$ for $0<|n|<L$ and
\[ V_1\cup\dots\cup V_i \subset \bigcup_{|n|<L}S^n U_i.\]
Thus $U=U_a$ satisfies the required properties.
\end{proof}

Let $L$ be a positive integer, and
$U\subset Z$ a clopen subset given by Lemma \ref{lemma: clopen marker}.
For each point $x\in Z$ we define $C(x)\subset \mathbb{Z}^k$ as the set of $n\in \mathbb{Z}^k$ satisfying
$S^nx\in U$.
From (\ref{eq: clopen marker}), this is syndetic (coarsely dense) in $\mathbb{Z}^k$.
From $U\cap S^nU=\emptyset$ $(0<|n|<L)$ we have $|n-m|\geq L$ for
any distinct $n$ and $m$ in $C(x)$.
We have
\begin{equation} \label{eq: equivariance of C(x)}
  C(S^n x) = -n + C(x).
\end{equation}
The set $C(x)$ depends continuously on $x$; for any $x\in Z$ and
any finite set $\Omega\subset \mathbb{Z}^k$ if $y\in Z$ is sufficiently close to $x$ then
$C(y)\cap \Omega = C(x)\cap \Omega$.

We consider the Voronoi decomposition associated with $C(x)$.
We define a bounded convex polytope $V(x,n)\subset \mathbb{R}^k$ for each
$n\in C(x)$ by
\[ V(x,n) = \{u\in \mathbb{R}^k|\, \forall m\in C(x): |u-n|\leq |u-m|\}.\]
This contains the closed ball $B_{L/2}(n)$ of radius $L/2$ centered at $n$.
These form a tiling:
\[ \mathbb{R}^k = \bigcup_{n\in C(x)} V(x,n).\]
The tiles
$V(x,n)$ are locally constant; for any $x\in Z$ and $n\in C(x)$ if $y\in Z$ is
sufficiently close to $x$ then $n\in C(y)$ and $V(y,n)=V(x,n)$.
We set $V^{\mathbb{Z}}(x,n) = \mathbb{Z}^k\cap V(x,n)$.

\begin{lemma}\label{lemma: voronoi cells are large}
For any $R>0$ we can choose $L$ sufficiently large so that all $V(x,n)$ satisfy
\[ \frac{|\partial^{\mathbb{Z}}_R V^\mathbb{Z}(x,n)|}{|V^{\mathbb{Z}}(x,n)|} <\frac{1}{R}.\]
\end{lemma}
\begin{proof}
For simplicity of the notation, we write $V=V(x,n)$.
For a positive number $c$ we define $cV$ by
\[ cV = \{n+c(u-n)\in \mathbb{R}^k|\, u\in V\}.\]
When $c<1$, $cV$ is contained in $V$. When $c>1$, $cV$ contains $V$.
\begin{claim} \label{claim: similarity and V}
If $c<1$ then all points $p$ in $cV$ satisfy $B_{(1-c)L/2}(p)\subset V$.
If $c>1$ then all points $p$ outside of $cV$ satisfy $B_{(c-1)L/2}(p)\cap V=\emptyset$.
\end{claim}
\begin{proof}
We consider the case $c<1$. The case $c>1$ is similar.
Pick $p\in cV$ and let $q\in \partial V$ be any point in the boundary $\partial V$.
Let $F$ be a facet ($(k-1)$-dimensional face) of $V$ containing $q$.
Let $cF = \{n+c(u-n)|\, u\in F\}$ be the corresponding facet of $cV$, and $H$ the
hyperplane containing $cF$.
Since $B_{L/2}(n)\subset V$, the distance between $H$ and $F$ is at least $(1-c)L/2$.
The line segment between $p$ and $q$ must intersects with $H$ because
$p$ and $q$ are located on different sides of $H$.
Hence $|p-q|\geq d(H,F) \geq (1-c)L/2$. This implies $B_{(1-c)L/2}(p)\subset V$.
\end{proof}

Now that the claim has been established we can finish the proof of Lemma~\ref{lemma: voronoi cells are large}.
Let $c=1-2\sqrt{k}/L$ and $p\in cV$.
Take $m\in \mathbb{Z}^k$ satisfying $p\in m+[0,1)^k$.
By $(1-c)L/2=\sqrt{k}$ and Claim \ref{claim: similarity and V}
\[ m\in  B_{\sqrt{k}}(p)\subset V.\]
So $m\in V^\mathbb{Z} = V\cap \mathbb{Z}^k$.
This means that
\[  \bigcup_{m\in V^{\mathbb{Z}}} (m+[0,1)^k)\supset cV.\]
Then $|V^\mathbb{Z}|\geq \vol(cV) = c^k \vol(V) = (1-2\sqrt{k}/L)^k \vol(V)$ where $\vol(V)$ is the $k$-dimensional Lebesgue
measure of $V$.

Set $c_1 = 1-2(R+\sqrt{k})/L$ and $c_2 = 1+2(R+\sqrt{k})/L$.
Note $(1-c_1)L/2 = (c_2-1)L/2= R+\sqrt{k}$.
Take $m\in \partial^{\mathbb{Z}}_R V^\mathbb{Z}$. Then for every point $p$ in $m+[0,1)^k$ the ball
$B_{R+\sqrt{k}}(p)$ has non-empty intersections both with $V$ and $\mathbb{R}^k\setminus V$.
This implies that $m+[0,1)^k$ is contained in $c_2V\setminus c_1V$ by Claim \ref{claim: similarity and V}.
Hence
\[ \bigcup_{m\in \partial^\mathbb{Z}_R V^\mathbb{Z}} (m+[0,1)^k) \subset c_2V\setminus c_1V.\]
Therefore $|\partial^\mathbb{Z}_R V^\mathbb{Z}| \leq \vol(c_2V\setminus c_1V) = (c_2^k-c_1^k)\vol(V)$.
As a conclusion,
\[ \frac{|\partial^\mathbb{Z}_R V^\mathbb{Z}|}{|V^\mathbb{Z}|} \leq
   \frac{(1+2(R+\sqrt{k})/L)^k-(1-2(R+\sqrt{k})/L)^k}{(1-2\sqrt{k}/L)^k}.\]
If $L$ is sufficiently large then the right-hand-side is smaller than $1/R$.
\end{proof}

\begin{proof}[Proof of Proposition \ref{prop: main proposition for embedding of zero dimensional extension}]

Let $\delta>0$ and $f\in C(X,K)$.
We choose $0<\varepsilon<\delta$ such that
\begin{equation} \label{eq: choice of epsilon for embdding in zero dimensional case}
   d(x,y)<\varepsilon\Longrightarrow |f(x)-f(y)|<\delta.
\end{equation}
Using $\mdim(X,T)<D/2$ and the definition of mean dimension in (\ref{eq: mean dimension by Folner sequence}),
we can find $R>0$ such that if a finite subset $\Omega\subset \mathbb{Z}^k$ satisfies
\[ \frac{|\partial^{\mathbb{Z}}_R\Omega|}{|\Omega|} < \frac{1}{R}\]
then we have
\begin{equation} \label{eq: condtion on R and widim}
   \frac{\widim_\varepsilon(X,d_{\Omega})}{|\mathrm{int}^{\mathbb{Z}}_1\Omega|} < \frac{D}{2}.
\end{equation}
Here recall $\mathrm{int}^{\mathbb{Z}}_1\Omega = \Omega\setminus \partial^{\mathbb{Z}}_1\Omega$.

We define an equivalence relation between finite subsets of $\mathbb{Z}^k$ by
\[ \Omega_1\sim \Omega_2 \Longleftrightarrow \exists a\in \mathbb{Z}^k: \Omega_2=a+\Omega_1.\]
We set
\[ \mathcal{A} := \left\{\Omega\subset \mathbb{Z}^k\text{: finite set}|\,
   \frac{|\partial_R^\mathbb{Z}\Omega|}{|\Omega|} < \frac{1}{R}\right\}/ \sim.\]
Choose $\Omega_1, \Omega_2,\dots \subset \mathbb{Z}^k$ so that
\[ \mathcal{A} = \{[\Omega_1],[\Omega_2],\dots\}, \quad \Omega_i \not\sim \Omega_j\, (i\neq j).\]
For each $\Omega_i$ we consider
\[ I_f|_{\mathrm{int}_1^\mathbb{Z}\Omega_i}:X\to K^{\mathrm{int}_1^\mathbb{Z}\Omega_i},
   \quad x\mapsto (f(T^nx))_{n\in \mathrm{int}_1^\mathbb{Z}\Omega_i}.\]
This satisfies (by (\ref{eq: choice of epsilon for embdding in zero dimensional case}))
\begin{multline*} d_{\Omega_i}(x,y)<\varepsilon\Longrightarrow\\
  \norm{I_f(x)|_{\mathrm{int}_1^\mathbb{Z}\Omega_i}-I_f(y)|_{\mathrm{int}_1^\mathbb{Z}\Omega_i}}_\infty
  \stackrel{\mathrm{def}}{=} \max_{n\in \mathrm{int}^\mathbb{Z}_1\Omega_i}|f(T^nx)-f(T^ny)|< \delta.
\end{multline*}
Choose an $\varepsilon$-embedding $p_i:(X,d_{\Omega_i})\to P_i$ such that $P_i$ is a simplicial complex of
dimension $\widim_\varepsilon(X,d_{\Omega_i}) < (D/2)|\mathrm{int}^\mathbb{Z}_1\Omega_i|$ (by (\ref{eq: condtion on R and widim})).
Then by Lemmas \ref{lemma: approximation by linear map} and \ref{lemma: embedding of simplicial complex}
we can find a linear \textit{embedding} $g_i:P_i\to K^{\mathrm{int}^\mathbb{Z}_1\Omega_i}$ satisfying
\[ \norm{g_i(p_i(x))-I_f(x)|_{\mathrm{int}_1^\mathbb{Z}\Omega_i}}_\infty < \delta.\]
Set $F_i = g_i\circ p_i:X\to K^{\mathrm{int}_1^\mathbb{Z}\Omega_i}$.
This is an $\varepsilon$-embedding with respect to~$d_{\Omega_i}$.

By Lemma \ref{lemma: voronoi cells are large} we can choose $L>0$ so that all Voronoi tiles $V(x,n)$ $(x\in Z)$ satisfy
$|\partial_R^\mathbb{Z} V^\mathbb{Z}(x,n)|/|V^\mathbb{Z}(x,n)|<1/R$.
We want to define a perturbation $g:X\to K$ of $f$.
Roughly speaking we define it by painting the functions $F_i$ inside the Voronoi tiles.
Take $x\in X$.
We look at the decomposition $\mathbb{Z}^k = \bigcup_{n\in C(\pi(x))}V^\mathbb{Z}(\pi(x),n)$
associated with $\pi(x)\in Z$, and we ask which tile contains the origin $0$.
There are two cases:
\begin{enumerate}
   \item  For any $n\in C(\pi(x))$ the origin is not contained in $\mathrm{int}_1^\mathbb{Z} V^\mathbb{Z}(\pi(x),n)$.
   \item  There uniquely exists $n\in C(\pi(x))$ satisfying $0\in \mathrm{int}_1^\mathbb{Z} V^\mathbb{Z}(\pi(x),n)$.
\end{enumerate}
In Case (1) we set $g(x)=f(x)$.
In Case (2) we proceed as follows.
The Voronoi tile $V(\pi(x),n)$ satisfies
$|\partial^\mathbb{Z}_R V^{\mathbb{Z}}(\pi(x),n)|/|V^{\mathbb{Z}}(\pi(x),n)|<1/R$.
Hence there uniquely exist $i\geq 1$ and $a\in \mathbb{Z}^k$ satisfying $V^\mathbb{Z}(\pi(x),n)=a+\Omega_i$.
Then we set
\[ g(x) = F_i(T^a x)_{-a}.\]
Here $-a\in \mathrm{int}_1^\mathbb{Z}\Omega_i$ because $0\in \mathrm{int}_1^\mathbb{Z} V^\mathbb{Z}(\pi(x),n)$.
The Voronoi tile $V(\pi(x),n)$ is locally constant with respect to $x$.
Hence the map $g$ becomes continuous.
(The above $i$ and $a$ are also locally constant.)
It satisfies $|g(x)-f(x)|<\delta$ because
\[ |F_i(T^a x)_{-a}-f(x)|\leq \norm{F_i(T^a x)-I_f(T^a x)|_{\mathrm{int}_1^\mathbb{Z}\Omega_i}}_\infty <\delta.\]

\begin{claim} \label{claim: description of I_g zero dimensional case}
Let $x\in X$ and $n\in C(\pi(x))$.
Take $i\geq 1$ and $a\in \mathbb{Z}^k$ satisfying $V^\mathbb{Z}(\pi(x),n)=a+\Omega_i$.
Then $I_g(x)|_{\mathrm{int}_1^\mathbb{Z}V^\mathbb{Z}(\pi(x),n)} = F_i(T^a x)$.
Here we naturally consider $I_g(x)|_{\mathrm{int}_1^\mathbb{Z} V^\mathbb{Z}(\pi(x),n)}$ as an element of
$K^{\mathrm{int}_1^\mathbb{Z}\Omega_i}$ through $\mathrm{int}_1^\mathbb{Z} V^\mathbb{Z}(\pi(x),n)=a+\mathrm{int}_1^\mathbb{Z}\Omega_i$.
\end{claim}
\begin{proof}
Take $m\in \mathrm{int}_1^\mathbb{Z}V^\mathbb{Z}(\pi(x),n)= a+\mathrm{int}_1^\mathbb{Z}\Omega_i$.
By (\ref{eq: equivariance of C(x)})
\[
V^\mathbb{Z}(\pi(T^m x),n-m) = -m+V^\mathbb{Z}(\pi(x),n) = -m+a+\Omega_i.
\]
Hence the origin is contained in $\mathrm{int}_1^\mathbb{Z} V^\mathbb{Z}(\pi(T^m x),n-m)$, and
\[
g(T^mx) = F_i(T^{-m+a}(T^m x))_{m-a} = F_i(T^a x)_{m-a}.
\]
When $m$ runs over $\mathrm{int}_1^\mathbb{Z} V^\mathbb{Z}(\pi(x),n)$, $m-a$ runs over $\mathrm{int}_1^\mathbb{Z}\Omega_i$.
Thus we get $I_g(x)|_{\mathrm{int}_1^\mathbb{Z}V^\mathbb{Z}(\pi(x),n)} = F_i(T^a x)$.
\end{proof}

We want to prove that
the map $(I_g,\pi):X\to K^{\mathbb{Z}^k}\times Z$ is a $\delta$-embedding with respect to $d$.
Suppose $(I_g(x),\pi(x))=(I_g(y),\pi(y))$ for some $x,y\in X$.
Set $z=\pi(x)=\pi(y)$, and take $n\in C(z)$ satisfying $0\in V^\mathbb{Z}(z,n)$.
We can find $i\geq 1$ and $a\in \mathbb{Z}^k$ satisfying $V^\mathbb{Z}(z,n)=a+\Omega_i$.
(Note that this implies $-a\in \Omega_i$.)
By Claim \ref{claim: description of I_g zero dimensional case}
\[ I_g(x)|_{\mathrm{int}_1^\mathbb{Z} V^\mathbb{Z}(z,n)} = F_i(T^a x)
   = I_g(y)|_{\mathrm{int}_1^\mathbb{Z} V^\mathbb{Z}(z,n)} = F_i(T^a y).\]
Since $F_i$ is an $\varepsilon$-embedding with respect to $d_{\Omega_i}$, we get
$d_{\Omega_i}(T^a x,T^a y)<\varepsilon$.
Then $-a\in \Omega_i$ implies $d(x,y)<\varepsilon<\delta$.
\end{proof}

\section{Adding one dimension} \label{section: adding one dimension}

The key idea in Section \ref{section: proof of embedding and symbolic system}
is the use of the Voronoi tiling.
The purpose of this section is to present a modified version of Voronoi tiling
technique.
This has a wider applicability and will be used in the rest of the paper.

  Suppose a system $(X,\mathbb{Z}^k,T)$ has the marker property (cf.\ Definition~\ref{def: marker}).
Let $M$ be a positive integer.
From the marker property there exist an integer $L\geq M$ and
a continuous function $\phi:X\to [0,1]$ so that
\begin{itemize}
 \item If $\phi(x)>0$ at some $x\in X$, then $\phi(T^nx) = 0$ for all non-zero $n\in \mathbb{Z}^k$
  with $|n|<M$.
 \item For any $x\in X$ there exists $n\in \mathbb{Z}^k$ satisfying $|n|<L$ and  $\phi(T^nx)=1$.
\end{itemize}
\begin{proof}
Let $U\subset X$ be an open set satisfying the definition of the marker property.
Since $X$ is compact, there exists a compact set $K\subset U$ and an integer $L\geq M$
satisfying
\[ X= \bigcup_{|n|<L} T^{-n}K.\]
Take a continuous function $\phi:X\to [0,1]$ satisfying
$\phi=1$ on $K$ and $\supp\, \phi\subset U$.
\end{proof}

Now we introduce the key technique: adding one dimension.
Take a point $x\in X$, and we consider the Voronoi tiling in $\mathbb{R}^{k+1}$ associated with the set
\[ \{(n,1/\phi(T^nx))|\, n\in \mathbb{Z}^k: \phi(T^nx)\neq 0\}. \]
We define the Voronoi cell
$V(x,n)\subset \mathbb{R}^{k+1}$ with the Voronoi center $(n,1/\phi(T^nx))$ by
\[ V(x,n) = \left\{u\in \mathbb{R}^{k+1}\middle|\parbox{3.15in}{$\forall m\in \mathbb{Z}^k:$\\ \hspace*{\fill} $|u-(n,1/\phi(T^nx))| \leq |u-(m,1/\phi(T^mx))|$}\right\}.\]
Here $|\cdot|$ is the standard Euclidean norm.
For $n\in \mathbb{Z}^k$ with $\phi(T^nx)=0$ we set $V(x,n) = \emptyset$.
The space $\mathbb{R}^{k+1}$ is decomposed into these convex subsets:
\[  \mathbb{R}^{k+1} = \bigcup_{n\in \mathbb{Z}^{k}} V(x,n).\]
We choose a real number $H\geq (L+\sqrt{k})^2$ and
set
\[
W(x,n) = \pi_{\R^k}\left(V(x,n)\cap (\mathbb{R}^k\times \{-H\})\right),
\]
with $\pi_{\R^k}$ denoting the projection to the first $k$ coordinates.

Then $\mathbb{R}^k$ is decomposed into these $W(x,n)$:
\[ \mathbb{R}^k = \bigcup_{n\in \mathbb{Z}^k} W(x,n).\]
This construction is naturally $\mathbb{Z}^k$-equivariant:
\[W(T^m x,n-m)=-m+W(x,n).\]
For each $n\in \mathbb{Z}^k$ the polytope $W(x,n)$ depends continuously on $x$ in the following sense:
Suppose $W(x,n)$ has a non-empty interior. For any $\varepsilon>0$ if $y\in X$ is sufficiently
close to $x$ then the Hausdorff distance between $W(x,n)$ and $W(y,n)$ are smaller than $\varepsilon$.
We gather basic properties of the Voronoi tiling in the next lemma.
\begin{lemma}  \label{lemma: Voronoi tiling}
Let $n\in \mathbb{Z}^k$ and $a\in \mathbb{R}^k$.
\begin{enumerate}
\item
If $\phi(T^nx)>0$, then \[B_{M/2}(n,1/\phi(T^nx))\subset V(x,n).\]
Here $B_{M/2}({\cdot})$ denotes the ball of radius $M/2$
with respect to the Euclidean norm.
\item If $W(x,n)$ is non-empty, then $1\leq 1/\phi(T^nx)\leq 2$.
\item If $(a,-H) \in V(x,n)$, i.e.\ $a\in W(x,n)$, then $|a-n|<L+\sqrt{k}$.
\item
Let $s>1$ and $r>0$.
We can choose $M$ sufficiently large depending on $s,r$ so that
if $(a,-sH)\in V(x,n)$ then \[B_r(a/s + (1-1/s)n)\subset W(x,n).\]
Note that the choice of $L,H$ depends on $M$.
\end{enumerate}
\end{lemma}

Taking $s \approx 1$ and $r$ large, property (4) of Lemma~\ref{lemma: Voronoi tiling} implies that if $V(x,n)$ intersects $\mathbb{R}^k\times \{-sH\}$ then
$W(x,n)$ contains a large ball.

\begin{proof}
(1) Since $\phi(T^nx)>0$, we have $\phi(T^{n+i}x)=0$ for $0<|i|<M$.
Hence for any $m\neq n$ with $\phi(T^mx)> 0$ we have $|n-m|\geq M$.
For
$u\in B_{M/2}(n,1/\phi(T^nx))$
\begin{multline*}
 |u-(m,1/\phi(T^mx))|\geq\\ |(n,1/\phi(T^nx))-(m,1/\phi(T^mx))| - |u-(n,1/\phi(T^nx))|
   \\
   \geq M/2.
 \end{multline*}

(2) and (3): Suppose $(a,-H) \in V(x,n)$ and set $t=1/\phi(T^nx)\geq 1$.
Let $b\in \mathbb{Z}^k$ be the nearest integer point to $a$.
There exists $|i|<L$ satisfying $\phi(T^{b+i}x)=1$. Set $m=b+i$.
We have $|a-m|< L+\sqrt{k}$ and
\[ |(a,-H)-(n,t)|\leq |(a,-H)-(m,1)| < \sqrt{(L+\sqrt{k})^2+(H+1)^2}.\]
We get $|a-n|<L+\sqrt{k}$ (by $t\geq 1$) and $(H+t)^2 < (L+\sqrt{k})^2 +(H+1)^2$.
Recall $H\geq (L+\sqrt{k})^2$. Then
\[ t \leq 1+\frac{(L+\sqrt{k})^2+1}{2H} \leq 2.\]
This proves (2) and (3).

(4)
The same argument as above shows $|a-n|< L+\sqrt{k}$. Set $t=1/\phi(T^n x)$.
For any $u\in \mathbb{R}^k$ with $|u-n|\leq M/2$ the point $(u,t)$ is contained in $V(x,n)$ by (1).
Consider the line $\ell$ between $(a,-sH)$ and $(u,t)$.
By the convexity of $V(x,n)$, the intersection between $\ell$ and $\mathbb{R}^k\times \{-H\}$
is contained in $W(x,n)$. It follows that
\begin{equation}  \label{eq: ball in W(x,n)}
   B_{\frac{(s-1)HM}{2(sH+t)}}\left(a+\frac{(s-1)H}{sH+t}(n-a)\right) \subset W(x,n).
\end{equation}

The radius of the ball in the left hand side of (\ref{eq: ball in W(x,n)}) goes to infinity as $M\to \infty$ because
\[ \frac{(s-1)HM}{2(sH+t)} = \frac{(s-1)M}{2(s+t/H)} \]
(recall $H$ depends on $M$, but as $M\to\infty$ so do both $L$ and $H$).

Our choice of paramters ensures that the center of the ball in (\ref{eq: ball in W(x,n)}) is close to $a/s+(1-1/s)n$. Indeed,
\begin{multline*}
   \left|\frac{a}{s} +\left(1-\frac{1}{s}\right)n -\left(a+\frac{(s-1)H}{sH+t}(n-a)\right)\right|=\\
   \begin{aligned}
   &=      \left|\left(1-\frac{1}{s}-\frac{(s-1)H}{sH+t}\right)(n-a)\right| \\
   &\leq   \left|1-\frac{1}{s}-\frac{s-1}{s+t/H}\right|(L+\sqrt{k}) \\
   &=       \left|\left(1-\frac1s\right)\cdot \frac{t}{H}\cdot\frac{1}{s+t/H}\right|(L+\sqrt{k}) \\
   &\leq    \frac{4(L+\sqrt{k})}{H} \leq \frac{4}{L+\sqrt{k}}.
   \end{aligned}
\end{multline*}
Here we have used $|n-a|<L+\sqrt{k}$, $1\leq t\leq 2$ and $H\geq (L+\sqrt{k})^2$.
Since $L\geq M$, this goes to zero as $M\to \infty$.
Therefore if $M$ is sufficiently large then $B_r(a/s+(1-1/s)n)$ is contained in $W(x,n)$.
\end{proof}

Let $A\subset \mathbb{R}^k$.
We denote by $\partial A$ and $\mathrm{int} A = A\setminus \partial A$ the standard boundary and interior of $A$.
For $E>0$
we define $\partial_E A$ as the set of $u\in \mathbb{R}^k$
satisfying $B_E(u)\cap A\neq \emptyset$ and $B_E(u)\cap (\mathbb{R}^k\setminus A)\neq \emptyset$.
Set $\mathrm{int}_E A =A\setminus \partial_E A$.
For $x\in X$, we define a set $\partial(x,E)$ in $\mathbb{R}^k$ by
\[ \partial(x,E) = \bigcup_{n\in \mathbb{Z}^k}\partial_E W(x,n).\]
The next lemma is crucial in the proofs of Theorems \ref{thm: zero mean dimension and topological entropy}
and \ref{thm: metric mean dimension is equal to mean dimension}.
\begin{lemma}\label{lemma: most tiles are large}
For any $\varepsilon>0$ and $E>0$ if we choose $M$ sufficiently large then
\[ \limsup_{R\to \infty} \frac{1}{\vol(B_R)}\sup_{x\in X}\vol(\partial(x,E)\cap B_R) < \varepsilon.\]
Here $B_R=\{u \in \mathbb{R}^k|\, |u|\leq R\}$, and $\vol(\cdot)$ is the $k$-dimensional Lebesgue measure.

\end{lemma}
\begin{proof}
Choose $s>1$ with $1-s^{-k}<\varepsilon$.
By Lemma \ref{lemma: Voronoi tiling} (4), we can choose $M$ so that
\[ (a,-sH)\in V(x,n)\Longrightarrow B_{E}\left(\frac{a}{s}+\left(1-\frac{1}{s}\right)n\right)\subset W(x,n).\]
For $n\in \mathbb{Z}^k$ and $x\in X$ we set
\[W'(x,n)=\pi_{\R^k}\left((\mathbb{R}^k\times \{-sH\})\cap V(x,n)\right).\]
For each $x\in X$ these $W'(x,n)$ $(n\in \mathbb{Z}^k)$ form a tiling of $\mathbb{R}^k$.
For any $a\in W'(x,n)$, the point $a/s+(1-1/s)n$ is contained in
$\mathrm{int}_E W(x,n)$. Hence
\begin{align*}
\vol(\mathrm{int}_E W(x,n))&\geq \vol\left(\left\{\frac{a}{s}+\left(1-\frac{1}{s}\right)n\middle| a\in W'(x,n)\right\}\right)\\
&   = s^{-k}\vol(W'(x,n)).
\end{align*}

Let $R>0$. As in the proof of Lemma \ref{lemma: Voronoi tiling}, all points $a\in W'(x,n)$ satisfy
$|a-n|<L+\sqrt{k}$.
Hence if $W'(x,n)$ has a non-empty intersection with the ball $B_{R-2L-2\sqrt{k}}$
then $|n|< R-L-\sqrt{k}$ and $W(x,n)\subset B_R$.
Therefore when $n\in \mathbb{Z}^k$ runs over $\{n|\, W(x,n)\subset B_R\}$, the sets $W'(x,n)$ cover the ball $B_{R-2L-2\sqrt{k}}$.
Thus
\begin{align*}
 \sum_{n\in\Z^k:W(x,n)\subset B_R}\!\!\vol(\mathrm{int}_E W(x,n)) &\geq
  s^{-k}\sum_{n\in\Z^k:W(x,n)\subset B_R}\!\!\vol(W'(x,n))\\
  &\geq s^{-k}\vol(B_{R-2L-2\sqrt{k}}).
\end{align*}
Hence
\begin{align*}
\vol(\partial(x,E)\cap B_R)&\leq \vol(B_R)-\vol\Bigl(\bigcup_{n:W(x,n)\subset B_R}\!\!\vol(\mathrm{int}_E W(x,n))\Bigr)\\
   &\leq \vol(B_R)-s^{-k}\vol(B_{R-2L-2\sqrt{k}}).
   \end{align*}
Thus
\[ \frac{1}{\vol(B_R)}\sup_{x\in X}\vol(\partial(x,E)\cap B_R) \leq
   1-s^{-k}\left(\frac{R-2L-2\sqrt{k}}{R}\right)^k, \]
\[  \limsup_{R\to \infty} \frac{1}{\vol(B_R)}\sup_{x\in X}\vol(\partial(x,E)\cap B_R)
    \leq 1-s^{-k} < \varepsilon.\]
\end{proof}

A main trick in the above proof was \textit{going down to $\mathbb{R}^k\times \{-sH\}$ from $\mathbb{R}^k\times \{-H\}$}.
This idea will be used again in Section \ref{section: embedding problem}.

\section{Small boundary property: proof of Theorem \ref{thm: zero mean dimension and topological entropy}}
\label{section: small entropy factors}

We prove Theorem \ref{thm: zero mean dimension and topological entropy} in this section.
We need to introduce the notion small boundary property
defined in \cite{Lindenstrauss--Weiss, Lindenstrauss} and used implicitly by Shub--Weiss in \cite{Shub--Weiss} (a related, weaker condition was used in \cite{Lindenstrauss lowering} that allows for some periodic points).
Let $(X,\mathbb{Z}^k,T)$ be a dynamical system.
For a subset $A\subset X$ its \textbf{orbit capacity} $\ocap(A)$ is defined by
\[ \ocap(A) = \lim_{N\to \infty}\frac{1}{N^k}\sup_{x\in X}\sum_{n\in [N]} 1_{A}(T^nx).\]
This limit exists because of the sub-additivity:
\[ \sup_{x\in X}\sum_{n\in \Omega_1\cup \Omega_2}1_A(T^nx) \leq
   \sup_{x\in X}\sum_{n\in \Omega_1}1_A(T^nx) + \sup_{x\in X}\sum_{n\in \Omega_2}1_A(T^nx)
   \quad (\Omega_1,\Omega_2\subset \mathbb{Z}^k).\]
In particular
\begin{equation} \label{eq: ocap and inf}
 \ocap(A) = \inf_{N\geq 1}\left(\frac{1}{N^k}\sup_{x\in X}\sum_{n\in [N]} 1_{A}(T^nx)\right).
\end{equation}
From Lemma \ref{lemma: Ornstein--Weiss} we also have
\begin{equation} \label{eq: ocap and balls}
 \ocap(A) = \lim_{R\to \infty}\left(\frac{1}{|B_R\cap \mathbb{Z}^k|}\sup_{x\in X}\sum_{n\in B_R\cap \mathbb{Z}^k} 1_{A}(T^nx)\right).
\end{equation}
The system $X$ is said to have the \textbf{small boundary property} if
for any point $x\in X$ and any open set $U\ni x$ there is an open neighborhood $V\subset U$
of $x$ satisfying $\ocap (\partial V) =0$.
The importance of this notion is clarified by the following
(see Lindenstrauss \cite[Section 4]{Lindenstrauss lowering};
many ideas of this theorem were already presented by Shub--Weiss \cite{Shub--Weiss}).

\begin{theorem} \label{thm: SBP and small entropy factor}
If $X$ has the small boundary property, then for any $\varepsilon>0$ and any two distinct points
$x,y\in X$ there exists a factor $\pi:X\to Y$ such that $\pi(x)\neq \pi(y)$ and $h_{\mathrm{top}}(Y)<\varepsilon$.
\end{theorem}

Therefore, if $X$ has the small boundary property,
any two points can be distinguished by an \textit{arbitrarily small
entropy factor}.
The paper \cite{Lindenstrauss lowering} discussed only $\mathbb{Z}$-actions, but its argument can be easily generalized to
$\mathbb{Z}^k$-actions.
(For a sketch of the proof, see also \cite[p. 257]{Lindenstrauss}.)

\begin{corollary}
If $X$ has the small boundary property then it is isomorphic to the inverse limit of
finite entropy systems.
\end{corollary}
\begin{proof}
Let $\Delta$ be the diagonal of $X\times X$.
By Theorem \ref{thm: SBP and small entropy factor} there exist a countable open covering
$\{U_n\times V_n\}_{n\geq 1}$ of $X\times X\setminus \Delta$ and factors $\pi_n:X\to Y_n$ such that
$\pi_n(U_n)\cap \pi_n(V_n)=\emptyset$ and $h_{\mathrm{top}}(Y_n)<\infty$.
Define $X_n = (\pi_1\times \pi_2\times \dots\times \pi_n)(X)$ in $Y_1\times Y_2\times\dots \times Y_n$.
These $X_n$ naturally form an inverse system and its limit is isomorphic to $X$.
Every $X_n$ has finite topological entropy.
\end{proof}

We denote by $C(X)$ the Banach space of continuous functions $f:X\to \mathbb{R}$
with the uniform norm $\norm{\cdot}$.
The following is the main result of this section.
\begin{theorem} \label{thm: zero mean dimension and ocap}
Suppose $X$ has the marker property and $\mdim(X)=0$.
Then the set of continuous functions $f:X\to \mathbb{R}$ satisfying
\[ \ocap\left(\{x\in X|\, f(x)=0\}\right)=0\]
is a dense $G_\delta$ subset of $C(X)$.
\end{theorem}

The following corollary contains Theorem \ref{thm: zero mean dimension and topological entropy}.

\begin{corollary}\label{cor: zero mean dimension implies SBP}
Suppose $X$ has the marker property and $\mdim(X)=0$.
Then $X$ has the small boundary property.
In particular it is isomorphic to the inverse limit of finite entropy systems.
\end{corollary}
\begin{proof}
Take $x\in X$ and its open neighborhood $U$.
There is a continuous function $f:X\to \mathbb{R}$ such that
$f(x)=1$ and $f=-1$ over $X\setminus U$.
By Theorem \ref{thm: zero mean dimension and ocap}
we can find a continuous function $g:X\to \mathbb{R}$ such that $\norm{f-g} < 1$ and
$\{g=0\}$ has zero orbit capacity.
Set $V=\{g>0\}$.
We have $x\in V\subset U$ and $\ocap(\partial V)=0$.
\end{proof}

The rest of this section consists of the proof of Theorem \ref{thm: zero mean dimension and ocap}.
Let $(X,\mathbb{Z}^k,T)$ be a dynamical system.
From the equation (\ref{eq: ocap and inf}),
for any closed set $A\subset X$ and any $\varepsilon>0$
there is an open set $U\supset A$ satisfying
$\ocap(U)<\ocap(A)+\varepsilon$.
This implies the next lemma.
\begin{lemma}\label{lemma: ocap stable}
For any $c>0$ the set
\[ \{f\in C(X)|\, \ocap(\{f=0\})<c\}\]
is open in $C(X)$.
\end{lemma}
\begin{proof}
Suppose a continuous function $f:X\to \mathbb{R}$ satisfies
$\ocap(\{f=0\})<c$.
Then there is an open set $U\supset \{f=0\}$ satisfying
$\ocap(U)<c$.
Let $\delta>0$ be the infimum of $|f(x)|$ over $x\in X\setminus U$.
Suppose a continuous function $g:X\to \mathbb{R}$ satisfies $\norm{f-g}<\delta$.
Then $\{g=0\}\subset U$. Hence $\ocap(\{g=0\})<c$.
\end{proof}
It follows from Lemma~\ref{lemma: ocap stable} that the set
\begin{multline*}
 \{f\in C(X)|\, \ocap(\{f=0\})=0\} = \\
   \bigcap_{n\geq 1} \{f\in C(X)|\, \ocap(\{f=0\})<1/n\}
\end{multline*}
is a $G_\delta$ subset of $C(X)$.
Therefore Theorem~\ref{thm: zero mean dimension and ocap}
follows from the next proposition.

\begin{proposition}  \label{prop: main proposition for small entropy factors}
Suppose $X$ has the marker property and $\mdim(X)=0$.
Let $\delta$ be a positive number, and $f:X\to \mathbb{R}$ a continuous function.
Then there exists a continuous function $g:X\to \mathbb{R}$ satisfying $\norm{f-g} <\delta$ and
$\ocap(\{g=0\})<\delta$.
\end{proposition}

We need one preparation:
\begin{lemma} \label{lemma: simplicial lemma for small entropy factor}
Let $n$ be a positive integer, and
$P$ a simplicial complex. Then almost every linear map $f:P\to \mathbb{R}^n$
satisfies
\[ |\{1\leq i\leq n|\, f(x)_i=0\}| \leq \dim P \quad (\forall x\in P).\]
\end{lemma}
\begin{proof}
Let $v_1,\dots,v_s$ be the vertices of $P$.
For almost every choice of vectors $u_1,\dots,u_s\in \mathbb{R}^n$ the following holds:
For every $A\subset \{1,2,\dots,n\}$ and $1\leq i_1<i_2<\dots<i_{|A|}\leq s$ the convex hull of
$u_{i_1}|_A,\dots,u_{i_{|A|}}|_A$ does not contain the origin in $\mathbb{R}^A$.
Suppose the vectors $u_1,\dots,u_s$ satisfy this condition, and define a linear map $f:P\to \mathbb{R}^n$
by $f(v_i)=u_i$.
Take any $x\in P$ and let $A=\{1\leq i\leq n|\, f(x)_i=0\}$.
Then $f(x)|_{A}=0$.
The assumption on $u_1,\dots,u_s$ implies
$|A|\leq \dim P$.
\end{proof}

\begin{proof}[Proof of Proposition \ref{prop: main proposition for small entropy factors}]
Let $d$ be a distance on $X$.
Take $0<\varepsilon<\delta/2$ such that
\[ d(x,y)<\varepsilon \Longrightarrow |f(x)-f(y)|<\delta.\]
Since $\mdim(X)=0$, we can find $N>0$ satisfying
\[ \frac{\widim_\varepsilon(X,d_{[N]})}{N^k}<\varepsilon.\]
We set $I_f(x)=(f(T^n x))_{n\in \mathbb{Z}^k}$ for $x\in X$. If $x,y\in X$ satisfy
$d_{[N]}(x,y)<\varepsilon$ then we have
\[ \norm{I_f(x)|_{[N]}-I_f(x)|_{[N]}}_\infty \stackrel{\mathrm{def}}{=}
   \max_{n\in [N]} |f(T^n x)-f(T^n y)| < \delta.\]
As in Section \ref{section: proof of embedding and symbolic system},
by Lemmas \ref{lemma: approximation by linear map} and \ref{lemma: simplicial lemma for small entropy factor}
we can construct a continuous map $F:X\to \mathbb{R}^{[N]}$ satisfying
\begin{itemize}
  \item  $\norm{F(x)-I_f(x)|_{[N]}}_\infty <\delta$ for all $x\in X$,
  \item  $|\{n\in [N]|\, F(x)_n=0\}| \leq \widim_\varepsilon(X,d_{[N]}) < \varepsilon N^k$ for all $x\in X$.
\end{itemize}
Set $E=\sqrt{k}(N+1)$.
We choose a sufficiently large positive integer $M$ and apply the Voronoi
tiling construction of Section \ref{section: adding one dimension} to $X$.
By Lemma \ref{lemma: most tiles are large} we can assume
\begin{equation} \label{eq: bound on the edge}
 \limsup_{R\to \infty} \frac{1}{\vol(B_R)}\sup_{x\in X}\vol(\partial(x,E)\cap B_R) < \varepsilon.
\end{equation}
Here recall that we have the tiling
$\mathbb{R}^k = \bigcup_{n\in \mathbb{Z}^k} W(x,n)$ and
$\partial(x,E) = \bigcup_{n\in \mathbb{Z}^k} \partial_E W(x,n)$.
By Lemma \ref{lemma: Voronoi tiling} (3), $\diam (W(x,n))<2L+2\sqrt{k}$.

We define a continuous function $g:X\to \mathbb{R}$ by painting $F$ in the tiles $W(x,n)$.
Let $\alpha:[0,\infty)\to [0,1]$ be a continuous function such that
$\alpha(0)=0$ and $\alpha(t)=1$ for $t\geq 1$.
Take $x\in X$. If $0\in \partial W(x,n)$ for some $n\in \mathbb{Z}^k$ then we set $g(x)=f(x)$.
Otherwise there uniquely exists $n\in \mathbb{Z}^k$ satisfying $0\in \mathrm{int} W(x,n)$.
Choosing $a\in \mathbb{Z}^k$ satisfying $a\equiv n \pmod N$ and $0\in a+[N]$,
we set
\begin{multline*} g(x) = \{1-\alpha(\dist(0,\partial W(x,n)))\}f(x) + \\
+\alpha(\dist(0,\partial W(x,n)))F(T^ax)_{-a},
\end{multline*}
with $\dist$ denoting the Euclidean distance.
Since $W(x,n)$ depends continuously on $x$, the function $g(x)$ is continuous.
It satisfies
$\norm{f-g}<\delta$ because $|F(T^ax)_{-a}-f(x)|\leq \norm{F(T^ax)-I_f(T^ax)|_{[N]}}_\infty <\delta$.
We will show $\ocap(\{g=0\})<\delta$.

\begin{claim} \label{claim: formula of I_g for small entropy factor}
Let $a,n\in \mathbb{Z}^k$ such that $a\equiv n \pmod N$ and
$a+[N]$ is contained in $\mathrm{int}_1 W(x,n)$. Then
\[ I_g(x)|_{a+[N]} = F(T^a x).\]
Here we naturally consider $I_g(x)|_{a+[N]} = (g(T^{a+n}x))_{n\in [N]}$ as a vector in $\mathbb{R}^{[N]}$.
\end{claim}
\begin{proof}
Take $m\in a+[N] \subset \mathrm{int}_1W(x,n)$.
Then $W(T^m x,n-m)=-m+W(x,n)$ and hence $0\in \mathrm{int}_1W(T^m x,n-m)$.
Since $0\in (a-m) +[N]$ and $a-m\equiv n-m \pmod N$,
\[ g(T^m x)= F(T^{a-m} T^m x)_{-a+m} = F(T^a x)_{-a+m}.\]
Thus we get $I_g(x)|_{a+[N]} = F(T^a x)$.
\end{proof}

\begin{claim} \label{claim: main claim for small entropy factor}
For all $x\in X$ and $n\in \mathbb{Z}^k$
\[  |\{m\in \mathbb{Z}^k\cap \mathrm{int}_{N\sqrt{k}}W(x,n)|\, I_g(x)_m=0\}| \leq \varepsilon\, \vol(W(x,n)).\]
\end{claim}
\begin{proof}
The set $\mathbb{Z}^k\cap \mathrm{int}_{N\sqrt{k}} W(x,n)$ is contained in the disjoint union of
$a+[N]$ where $a\in \mathbb{Z}^k$ satisfies $a\equiv n \pmod N$ and
$a+[N]\subset \mathrm{int}_{\sqrt{k}} W(x,n)$.
Take such $a+[N]$. Then $I_g(x)|_{a+[N]}=F(T^ax)$ by Claim \ref{claim: formula of I_g for small entropy factor}.
The second condition of $F$ implies
\[ |\{m\in a+[N]|\, I_g(x)_m=0\}| < \varepsilon \, N^k = \varepsilon \, \vol(a+[0,N)^k).\]
Summing up this estimate, we get the result.
\end{proof}

Continuing with the proof of Proposition~\ref{prop: main proposition for small entropy factors},
let $R>0$ and $x\in X$.
The number of $m\in \mathbb{Z}^k$ satisfying $|m|\leq R$ and $I_g(x)_m=0$ is bounded by
\begin{multline*}
 \left|\mathbb{Z}^k\cap B_R\cap \partial(x,N\sqrt{k})\right| + \\
 \begin{aligned}
 &+ \!\!\sum_{n:W(x,n)\cap B_R\neq \emptyset}\!\!
   \left|\left\{m\in \mathbb{Z}^k\cap \mathrm{int}_{N\sqrt{k}} W(x,n)\middle|\, I_g(x)_m=0\right\}\right| \\
   &\leq |\mathbb{Z}^k\cap B_R\cap \partial(x,N\sqrt{k})| +
 \varepsilon \sum_{n:W(x,n)\cap B_R\neq \emptyset} \vol(W(x,n))\\
 & \hspace{2in}{(\text{by Claim \ref{claim: main claim for small entropy factor}})}\\
   &\leq |\mathbb{Z}^k\cap B_R\cap \partial(x,N\sqrt{k})| + \varepsilon\, \vol(B_{R+2L+2\sqrt{k}})\\
 &\hspace{2in}(\text{by $\diam (W(x,n))< 2L+2\sqrt{k}$}).
\end{aligned}
\end{multline*}
Recall $E=\sqrt{k}(N+1)$. If $m\in \partial(x,N\sqrt{k})$, then $m+[0,1)^k\subset \partial(x,E)$.
Hence
\[\bigl|\mathbb{Z}^k\cap B_R\cap \partial(x,N\sqrt{k})\bigr|\leq \vol(\partial(x,E)\cap B_{R+\sqrt{k}}).\]
Therefore
\begin{multline*}
\frac{|\{m\in \mathbb{Z}^k\cap B_R|\, I_g(x)_m=0\}|}{|\mathbb{Z}^k\cap B_R|}\\
   \leq \frac{\vol(\partial(x,E)\cap B_{R+\sqrt{k}})}{|\mathbb{Z}^k\cap B_R|}
       + \varepsilon \frac{\vol(B_{R+2L+2\sqrt{k}})}{|\mathbb{Z}^k\cap B_R|}.
\end{multline*}	
Using the formula (\ref{eq: ocap and balls}) and the condition
(\ref{eq: bound on the edge}) we get
\begin{align*} \ocap(\{g=0\}) &= \lim_{R\to \infty}
   \frac{\sup_{x\in X}|\{m\in \mathbb{Z}^k\cap B_R|\, I_g(x)_m=0\}|}{|\mathbb{Z}^k\cap B_R|}\\
   &< \varepsilon + \varepsilon = 2\varepsilon <\delta.
   \end{align*}
Here we have used $|\mathbb{Z}^k\cap B_R|\sim \vol(B_R)$ as $R\to \infty$.
\end{proof}

\section{Metric mean dimension: proof of Theorem \ref{thm: metric mean dimension is equal to mean dimension}}
\label{section: metric mean dimension is equal to mean dimension}

We prove Theorem \ref{thm: metric mean dimension is equal to mean dimension} in this section.
Let $(X,d)$ be a compact metric space, and $\alpha$ an open covering of $X$.
We denote by $\mesh(\alpha,d)$ the supremum of $\diam\, U$ over $U\in \alpha$.
For $\varepsilon>0$ we define $A(X,\varepsilon,d)$ as the minimum cardinality of open coverings $\alpha$ of $X$
satisfying $\mesh(\alpha,d)<\varepsilon$.
Let $T: \mathbb{Z}^k\times X\to X$ be a continuous action.
We have a natural sub-additivity:
$\log A(X,\varepsilon,d_{\Omega_1\cup\Omega_2}) \leq \log A(X,\varepsilon,d_{\Omega_1})+\log A(X,\varepsilon,d_{\Omega_2})$.
We define
\begin{equation}  \label{eq: inf entropy at finite scale}
   S(X,\varepsilon,d) = \lim_{N\to \infty} \frac{1}{N^k}\log A(X,\varepsilon, d_{[N]})
   =  \inf_{N\geq 1}\frac{1}{N^k}\log A(X,\varepsilon, d_{[N]}).
\end{equation}
By Lemma \ref{lemma: Ornstein--Weiss} we also have
\begin{equation} \label{eq: S(X) and balls}
  S(X,\varepsilon,d) = \lim_{R\to \infty} \frac{1}{|B_R\cap \mathbb{Z}^k|}\log A(X,\varepsilon,d_{B_R\cap\mathbb{Z}^k}).
\end{equation}
We define the \textbf{metric mean dimension} $\mmdim(X,d)$ by
\[ \mmdim(X,d) = \liminf_{\varepsilon\to 0}\frac{S(X,\varepsilon,d)}{|\log \varepsilon|}.\]
This is a dynamical analogue of box-counting dimension in fractal geometry.
Metric mean dimension is always an upper bound on mean dimension (\cite[Theorem 4.2]{Lindenstrauss--Weiss}):
\begin{theorem} \label{thm: mdim leq mmdim}
\[ \mdim(X)\leq \mmdim(X,d).\]
\end{theorem}
The main question here is whether the equality holds for \textit{generic} distances $d$ or not.
We apply to this question a machinery developed in Section \ref{section: small entropy factors}.
Let $V$ be a Banach space (possibly infinite dimensional)
with a norm $\norm{\cdot}_V$, and $K\subset V$ a compact convex subset.
Let $(K^{\mathbb{Z}^k},\mathbb{Z}^k,\sigma)$ be the $\mathbb{Z}^k$-shift on $K^{\mathbb{Z}^k}$ where
\[ \sigma^n(x)=y, \quad y_m = x_{m+n}.\]
We define a distance $D$ on $K^{\mathbb{Z}^k}$ by
\[ D(x,y) = \sum_{n\in \mathbb{Z}^k}2^{-|n|}\norm{x_n-y_n}_V.\]
We study maps from $X$ to $K^{\mathbb{Z}^k}$.
Let $C(X,K)$ be the space of continuous maps from $X$ to $K$ equipped with the distance
$\norm{f-g}=\sup_{x\in X}\norm{f(x)-g(x)}_V$.
For a continuous map $f:X\to K$ we define $I_f:X\to K^{\mathbb{Z}^k}$ by
$I_f(x) = (f(T^nx))_{n\in \mathbb{Z}^k}$.
The image $I_f(X)$ is a subsystem of $K^{\mathbb{Z}^k}$.
The next theorem is the main result of this section.

\begin{theorem} \label{thm: mmdim leq mdim}
Suppose $X$ has the marker property.
Then for a dense $G_\delta$ subset of $f\in C(X,K)$
\[ \mmdim(I_f(X),D)\leq \mdim(X).\]
\end{theorem}

Assuming this theorem, we can easily prove Theorem \ref{thm: metric mean dimension is equal to mean dimension}.
\begin{proof}[Proof of Theorem \ref{thm: metric mean dimension is equal to mean dimension}]
Let $V=\ell^2(\mathbb{N})$ be the standard $\ell^2$-space with
$K=\{(u_n)_{n=1}^\infty\in \ell^2|\, 0\leq u_n\leq 1/n\}$.
Now $K$ is a compact convex subset homeomorphic to $[0,1]^{\mathbb{N}}$.
Generic continuous maps from compact metric spaces to $[0,1]^{\mathbb{N}}$ are topological embeddings.
This is well-known in classical dimension theory (Hurewicz--Wallman \cite[Theorem V4]{Hurewicz--Wallman});
it can be also deduced from Lemmas \ref{lemma: approximation by linear map} and
\ref{lemma: embedding of simplicial complex}.
So the map $I_f:X\to K^{\mathbb{Z}^k}$ becomes an embedding for generic $f\in C(X,K)$
because $f$ itself is an embedding.
Then by Theorem \ref{thm: mmdim leq mdim} we can find a continuous map $f:X\to K$ such that
$I_f$ is an embedding and $\mmdim(I_f(X),D)\leq \mdim(X)$.
Let $d$ be the pull-back of the distance $D$ by $I_f$.
We get $\mmdim(X,d)\leq \mdim(X)$.
The inequality $\mmdim(X,d)\geq \mdim(X)$ is always true.
Thus $\mmdim(X,d)=\mdim(X)$.
\end{proof}

Let $(X,\mathbb{Z}^k,T)$ be a dynamical system.
\begin{lemma}
For $\delta>0$ and $c>0$ the set of $f\in C(X,K)$ satisfying
\[ \exists 0<\varepsilon<\delta: \frac{S(I_f(X),\varepsilon,D)}{|\log \varepsilon|} < c\]
is open.
\end{lemma}
\begin{proof}
Suppose $f:X\to K$ and $0<\varepsilon<\delta$ satisfy
$S(I_f(X),\varepsilon, D)<c|\log \varepsilon|$.
Then there exist $N>0$ and an open covering $\alpha$ of $I_f(X)$ such that
$\mesh(\alpha,D_{[N]})<\varepsilon$ and $|\alpha| < \exp(cN^k|\log \varepsilon|)$.
For each $U\in \alpha$ we can find an open set $\tilde{U}\supset U$ of
$K^{\mathbb{Z}^k}$ with $\diam (\tilde{U},D_{[N]})< \varepsilon$.
Set $\tilde{\alpha} = \{\tilde{U}\}_{U\in \alpha}$.
If $g:X\to K$ is sufficiently close to $f$ then
$I_g(X)\subset \bigcup \tilde{\alpha}$.
Hence
\[ A(I_g(X),\varepsilon, D_{[N]})\leq |\tilde{\alpha}| < \exp(cN^k|\log \varepsilon|).\]
By (\ref{eq: inf entropy at finite scale}), we get $S(I_g(X),\varepsilon,D) < c|\log \varepsilon|$.
\end{proof}

By this lemma the set
\begin{equation*}
  \begin{split}
  &\{f\in C(X,K)|\, \mmdim(I_f(X),D) \leq \mdim(X)\} \\
   &= \bigcap_{n\geq 1}\left\{f\in C(X,K)\middle| \exists 0<\varepsilon<\frac{1}{n}:
     \frac{S(I_f(X),\varepsilon,D)}{|\log \varepsilon|} < \mdim(X)+\frac{1}{n}\right\}
  \end{split}
\end{equation*}
is a $G_\delta$ subset of $C(X,K)$.
Therefore Theorem \ref{thm: mmdim leq mdim} follows from the next proposition.

\begin{proposition} \label{prop: main proposition for metric mean dimension}
Suppose $X$ has the marker property.
Let $\delta$ be a positive number, and $f:X\to K$ a continuous map.
There exist $0<\varepsilon<\delta$ and a continuous map $g:X\to K$ satisfying $\norm{f-g}<\delta$ and
\[ S(I_g(X),\varepsilon,D) < (\mdim(X)+\delta)|\log \varepsilon|.\]
\end{proposition}

We need some preparations.
Let $Y\subset K^{\mathbb{Z}^k}$ be a closed invariant subset.
For $R>0$ we denote by $Y|_{B_R}$ the image of $Y$ under the natural projection
$K^{\mathbb{Z}^k}\to K^{B_R\cap \mathbb{Z}^k}$.
Let $\norm{\cdot}_\infty$ be the $\ell^\infty$-distance on $K^{B_R\cap \mathbb{Z}^k}$ defined by
$\norm{x-y}_\infty = \max_{n\in B_R\cap \mathbb{Z}^k}\norm{x_n-y_n}_V$.
\begin{lemma} \label{lemma for metric mean dimension}
There exists a universal constant $\kappa >0$ such that for every $\varepsilon>0$
\[ S(Y,\varepsilon,D)\leq \limsup_{R\to \infty}\frac{1}{\vol(B_R)}\log A(Y|_{B_R},\kappa \varepsilon,\norm{\cdot}_\infty).\]
\end{lemma}
\begin{proof}
Set $c=\sum_{n\in \mathbb{Z}^k}2^{-|n|}$, and take $L=L(\varepsilon)>0$ satisfying
\[ \sum_{n\in \mathbb{Z}^k, |n|>L}2^{-|n|}\diam K < \varepsilon/2.\]
Let $\pi:K^{\mathbb{Z}^k} \to K^{B_{R+L}\cap \mathbb{Z}^k}$ be the projection.
For any $n\in B_R\cap \mathbb{Z}^k$ and $x,y\in K^{\mathbb{Z}^k}$
\begin{equation*}
  \begin{split}
   D(\sigma^n(x),\sigma^n(y)) &= \sum_{m\in \mathbb{Z}^k}2^{-|m|}\norm{x_{n+m}-y_{n+m}}_V \\
   &\leq \sum_{|m|\leq L}2^{-|m|}\norm{x_{n+m}-y_{n+m}}_V + \sum_{|m|>L}2^{-|m|}\diam\, K \\
   &< c\norm{\pi(x)-\pi(y)}_\infty + \frac{\varepsilon}{2}.
  \end{split}
\end{equation*}
So $D_{B_R\cap \mathbb{Z}^k}(x,y)< c\norm{\pi(x)-\pi(y)}_\infty +\varepsilon/2$.
We set $\kappa = 1/(2c)$. Then for every $\varepsilon>0$
\[ A(Y,\varepsilon,D_{B_R\cap \mathbb{Z}^k}) \leq A(Y|_{B_{R+L}},\kappa \varepsilon,\norm{\cdot}_\infty).\]
Using (\ref{eq: S(X) and balls}) and $|B_R\cap \mathbb{Z}^k|\sim \vol(B_R)$ as $R\to \infty$,
\[ S(Y,\varepsilon,D) = \lim_{R\to \infty}\frac{1}{|B_R\cap \mathbb{Z}^k|}\log A(Y,\varepsilon,D_{B_R\cap \mathbb{Z}^k})
   \leq \limsup_{R\to \infty}\frac{1}{\vol(B_R)}\log A(Y|_{B_R},\kappa \varepsilon,\norm{\cdot}_\infty).\]
\end{proof}

\begin{lemma}\label{lemma: spaning set of polyhedron}
Let $W$ be a Banach space,
$P$ a simplicial complex, and
$f:P\to W$ a linear map.
Then for any $0<\varepsilon\leq 1$
\[ A(f(P),\varepsilon,\norm{\cdot}) \leq \mathrm{const}\cdot \varepsilon^{-\dim P},\]
where $\mathrm{const}$ is a positive constant depending only on $\diam f(P)$, $\dim P$ and the number of the simplices of $P$.
\end{lemma}
\begin{proof}
We can assume without loss of generality that $\diam f(P)\leq 1$ and $P$ is an $n$-dimensional simplex.
Then we can find $a,u_1,\dots,u_n\in W$ with $\norm{u_i}\leq 1$ such that
\[ f(P) = \{a+x_1 u_1+\dots+ x_n u_n|\, x_1,\dots,x_n\geq 0, \, x_1+\dots+x_n\leq 1\}.\]
Consider the following $(1+\lfloor 2n/\varepsilon\rfloor)^n$ points in $W$:
\[ a+x_1 u_1+\dots+x_n u_n\, \quad
\left(x_i=0,\frac{\varepsilon}{2n}, \frac{2\varepsilon}{2n},\dots, \left\lfloor \frac{2n}{\varepsilon}\right\rfloor \frac{\varepsilon}{2n}\right) .\]
Now $f(P)$ is covered by the open $\varepsilon/2$-balls around these points.
Hence
\[ A(f(K),\varepsilon,\norm{\cdot})\leq \left(1+\left\lfloor\frac{2n}{\varepsilon}\right\rfloor\right)^n
   \leq \frac{(2n+1)^n}{\varepsilon^n}.\]
\end{proof}

\begin{proof}[Proof of Proposition \ref{prop: main proposition for metric mean dimension}]
We can assume $\mdim(X)<\infty$.
Let $f\in C(X,K)$.
By Lemma \ref{lemma for metric mean dimension} it is enough to prove that for any $\delta>0$
there exist $0<\varepsilon<\delta$ and a continuous map $g:X\to K$ satisfying
$\norm{f-g}<\delta$ and
\[ \limsup_{R\to \infty}\frac{1}{\vol(B_R)}\log A(I_g(X)|_{B_R},\varepsilon,\norm{\cdot}_\infty) <
   |\log \varepsilon|(\mdim(X)+\delta).\]
Fix $0<\tau<\delta/3$ satisfying
\[ d(x,y) < \tau \Longrightarrow \norm{f(x)-f(y)}_V < \delta.\]
We choose $N>0$ satisfying
\[ \frac{1}{N^k}\widim_\tau(X,d_{[N]}) < \mdim(X,T)+\tau.\]
By Lemma \ref{lemma: approximation by linear map}
there exists a continuous map $F:X\to K^{[N]}$ such that
\begin{itemize}
  \item $\norm{F(x)-I_f(x)|_{[N]}}_\infty < \delta$ for all $x\in X$.
  \item The set $F(X)$ is contained in the image of a linear map
   from a $\widim_\tau(X,d_{[N]})$-dimensional simplicial complex to $K^{[N]}$.
\end{itemize}
From the second condition and Lemma \ref{lemma: spaning set of polyhedron}, we can find
$0<\varepsilon<\delta$ satisfying
\begin{equation}\label{eq: spanning set for F(X)}
 A(F(X),\varepsilon,\norm{\cdot}_\infty) < \varepsilon^{-N^k(\mdim(X)+\tau)}.
\end{equation}
We also assume (for the later convenience) $\log 2 < \tau |\log \varepsilon|$ and $|\log \varepsilon|> 1$.

Set $E=\sqrt{k}(N+1)$. As in Section \ref{section: small entropy factors} we choose a sufficiently
large integer $M$ and apply the Voronoi tiling construction in Section \ref{section: adding one dimension}
to $X$. From Lemma \ref{lemma: most tiles are large} we can assume
\begin{equation} \label{eq: edge effect for metric mean dimension}
 \limsup_{R\to \infty} \frac{1}{\vol(B_R)}\sup_{x\in X}\vol(\partial(x,E)\cap B_R)
   < \frac{\tau}{\log A(K,\varepsilon,\norm{\cdot}_V)}.
\end{equation}
Here $\mathbb{R}^k= \bigcup_{n\in \mathbb{Z}^k}W(x,n)$ and
$\partial(x,E) = \bigcup_{n\in \mathbb{Z}^k}\partial_E W(x,n)$.
We define $g:X\to K$ in the same way as in Section \ref{section: small entropy factors}.
Let $\alpha:[0,\infty)\to [0,1]$ be a continuous function such that
$\alpha(0)=0$ and $\alpha(t)=1$ for $t\geq 1$.
Take $x\in X$. If $0\in \partial W(x,n)$ for some $n\in \mathbb{Z}^k$ then we set $g(x)=f(x)$.
Otherwise there uniquely exists $n\in \mathbb{Z}^k$ satisfying $0\in \mathrm{int}W(x,n)$.
Choosing $a\in \mathbb{Z}^k$ satisfying $a\equiv n \pmod N$ and $0\in a+[N]$,
we set
\[ g(x) = \{1-\alpha(\dist(0,\partial W(x,n)))\}f(x) + \alpha(\dist(0,\partial W(x,n)))F(T^ax)_{-a}.\]
This satisfies $\norm{f(x)-g(x)}_V<\delta$.
We estimate $A(I_g(X)|_{B_R},\varepsilon,\norm{\cdot}_\infty)$ for $R\gg 1$.

For $R>0$ we define $\mathcal{C}_R$ as the set of subsets $C\subset B_R\cap \mathbb{Z}^k$ satisfying
\begin{itemize}
  \item $a+[N]\subset B_R$ for all $a\in C$,
  \item $(a+[N])\cap (b+[N])=\emptyset$ for any two distinct $a,b\in C$,
  \item $a+[N]$ $(a\in C)$ cover most of $B_R\cap \mathbb{Z}^k$ in the following sense:
 \[|B_R\cap \mathbb{Z}^k\setminus \bigcup_{a\in C}(a+[N])| <
 \frac{\tau \, \vol(B_R)}{\log A(K,\varepsilon,\norm{\cdot}_V)}.\]
\end{itemize}

\begin{claim} \label{claim for metric mean dimension}
If we choose $R$ sufficiently large, then for any $x\in X$ there exists $C\in \mathcal{C}_R$ such that
$I_g(x)|_{a+[N]}\in F(X)$ for all $a\in C$.
\end{claim}
\begin{proof}
Let $R>0$ and $x\in X$. For each $n\in \mathbb{Z}^k$ with $W(x,n)\subset B_R$
we define $C_n\subset \mathbb{Z}^k$ as the set of $a\in \mathbb{Z}^k$ satisfying
$a\equiv n\pmod N$ and $a+[N]\subset \mathrm{int}_{1} W(x,n)$.
We define $C$ as the union of $C_n$ over $n\in \mathbb{Z}^k$ with $W(x,n)\subset B_R$.
For every $a\in C$ we can prove $I_g(x)|_{a+[N]} = F(T^ax)\in F(X)$ as in Claim \ref{claim: formula of I_g for small entropy factor}.

The set $C$ obviously satisfies the first and second conditions in the definition of $\mathcal{C}_R$.
The problem is to confirm the third condition.
For each $n\in \mathbb{Z}^k$ with $W(x,n)\subset B_R$ the set
$\mathbb{Z}^k\cap \mathrm{int}_{N\sqrt{k}}W(x,n)$ is covered by $a+[N]$ $(a\in C_n)$.
Hence
\[ \bigcup_{n: W(x,n)\subset B_R} (\mathbb{Z}^k\cap \mathrm{int}_{N\sqrt{k}}W(x,n))
   \subset \bigcup_{a\in C}(a+[N]).\]
By $\diam W(x,n)< 2L+2\sqrt{k}$ (Lemma \ref{lemma: Voronoi tiling} (3)),
the union of $W(x,n)$ with $W(x,n)\subset B_R$ contains the ball
$B_{R-2L-2\sqrt{k}}$.
Therefore $B_R\cap \mathbb{Z}^k\setminus \bigcup_{a\in C}(a+[N])$ is contained in
\[ \{\mathbb{Z}^k\cap (B_R\setminus B_{R-2L-2\sqrt{k}})\}\cup
   \{\mathbb{Z}^k\cap B_{R-2L-2\sqrt{k}}\cap \partial(x,N\sqrt{k})\}.\]
The first term is small compared to $\vol(B_R)$:
\[ |\mathbb{Z}^k\cap (B_R\setminus B_{R-2L-2\sqrt{k}})| = O(R^{k-1}) = o(\vol(B_R)).\]
The second term is bounded by
\begin{align*} |\mathbb{Z}^k\cap B_{R-2L-2\sqrt{k}}\cap \partial(x,N\sqrt{k})| &\leq
   \vol(B_{R-2L-\sqrt{k}}\cap \partial(x,E)) \\
   &\leq \vol(B_R\cap \partial(x,E))
   \end{align*}
   (Recall $E=\sqrt{k}(N+1)$).
The last term can be estimated by (\ref{eq: edge effect for metric mean dimension}).
Thus if $R$ is sufficiently large (uniformly in $x\in X$) then
\[ |B_R\cap \mathbb{Z}^k\setminus \bigcup_{a\in C}(a+[N])| <
   \frac{\tau \, \vol(B_R)}{\log A(K,\varepsilon,\norm{\cdot}_V)}.\]
\end{proof}

For each $C\in \mathcal{C}_R$ we set
\[ Q_C = \left\{u\in K^{B_R\cap\mathbb{Z}^k}|\, \forall a\in C: u|_{a+[N]}\in F(X)\right\}.\]
Every $C\in \mathcal{C}_R$ satisfies $|C|\leq |B_R\cap\mathbb{Z}^k|/N^k$. Hence
\begin{equation*}
  \begin{split}
  \log A(Q_C,\varepsilon,\norm{\cdot}_\infty) &\leq |C|\cdot\log A(F(X),\varepsilon,\norm{\cdot}_\infty)\\
  &\qquad+|B_R\cap\mathbb{Z}^k\setminus \bigcup_{a\in C}(a+[N])|\cdot \log A(K,\varepsilon,\norm{\cdot}_V) \\
  &< \frac{|B_R\cap\mathbb{Z}^k|}{N^k}\log A(F(X),\varepsilon,\norm{\cdot}_\infty) + \tau\, \vol(B_R)\\
  &< |B_R\cap\mathbb{Z}^k|\cdot(\mdim(X)+\tau)|\log \varepsilon| + \tau\, \vol(B_R),
  \end{split}
\end{equation*}
where we have used \eqref{eq: spanning set for F(X)} to pass to the last line of the above inequality.
From Claim \ref{claim for metric mean dimension}, the set $I_g(X)|_{B_R}$ is contained in
the union of $Q_C$ over $C\in \mathcal{C}_R$ if $R$ is sufficiently large.
Hence for $R\gg 1$
\begin{multline*}
   \log A(I_g(X)|_{B_R},\varepsilon,\norm{\cdot}_\infty) < \log |\mathcal{C}_R| +\\
   + |B_R\cap\mathbb{Z}^k| \cdot(\mdim(X)+\tau)|\log \varepsilon| + \tau\, \vol(B_R).
\end{multline*}
Obviously $|\mathcal{C}_R|\leq 2^{|B_R\cap\mathbb{Z}^k|}$.
Recall that we assumed $\log 2< \tau |\log \varepsilon|$.
Hence
\[  \log A(I_g(X)|_{B_R},\varepsilon,\norm{\cdot}_\infty) <
    |B_R\cap\mathbb{Z}^k| \cdot(\mdim(X)+2\tau)|\log \varepsilon| + \tau\, \vol(B_R).\]
Since $|B_R\cap \mathbb{Z}^k|\sim \vol(B_R)$ as $R\to \infty$,
\begin{equation*}
  \begin{split}
  \limsup_{R\to \infty} \frac{1}{\vol(B_R)}\log A(I_g(X)|_{B_R},\varepsilon,\norm{\cdot}_\infty)
  &\leq (\mdim(X)+2\tau)|\log \varepsilon|  + \tau \\
  &< (\mdim(X)+3\tau)|\log \varepsilon| \\
  &\hspace{.95in} (\text{recall $|\log \varepsilon| > 1$})\\
  &< (\mdim(X)+\delta)|\log \varepsilon| \\
&\hspace{.95in}  (\text{recall $3\tau<\delta$}).
  \end{split}
\end{equation*}
\end{proof}

\section{A more difficult embedding theorem: Proof of Theorem \ref{thm: embedding theorem}}
\label{section: embedding problem}

Theorem \ref{thm: embedding theorem} follows from

\begin{theorem} \label{thm: embedding theorem for systems having marker property}
Let $D$ be a positive integer, and $(X,\mathbb{Z}^k,T)$ a dynamical system having the marker property.
If $\mdim(X) < D/2^{k+1}$, then for a dense $G_\delta$ subset of $f\in C(X,[0,1]^{2D})$ the map
\[ I_f:X\to ([0,1]^{2D})^{\mathbb{Z}^k}, \quad x\mapsto (f(T^nx))_{n\in \mathbb{Z}^k},\]
is an embedding.
\end{theorem}

Throughout this section, we set $K=[0,1]^D$.
We always assume that $(X,\mathbb{Z}^k,T)$ has the marker property and
$\mdim(X)<D/2^{k+1}$.
Let $d$ be a distance on $X$.
As in Section \ref{section: proof of embedding and symbolic system} the $G_\delta$ subset
\[ \bigcap_{n\geq 1}
  \left\{f\in C(X,K^2)\middle|\parbox{3in}{\centering $I_f:X\to (K^2)^{\mathbb{Z}^k}$ is a $(1/n)$-embedding with respect to $d$}\right\} \]
is equal to the set of $f\in C(X,K^2)$ such that $I_f$ is an embedding.
Therefore Theorem \ref{thm: embedding theorem for systems having marker property} follows from the next proposition.
\begin{proposition} \label{prop: main proposition for embedding}
Let $f=(f_1,f_2):X\to K^2$ be a continuous map. For any positive number $\delta$ there exists a continuous map
$g =(g_1,g_2):X\to K^2$ satisfying

\noindent
(1) $|f_i(x)-g_i(x)|<\delta$ for both $i=1,2$ and all $x\in X$,

\noindent
(2) the map $I_g:X\to (K^2)^{\mathbb{Z}^k}$ is a $\delta$-embedding with respect to the distance $d$.
\end{proposition}

The proof of this proposition occupies the rest of the section.

\subsection{Linear maps on simplicial complexes}

The purpose of this subsection is to prove Lemma \ref{lemma: simplicial complex lemma for embedding} below.
The argument here is elementary but notationally complicated.

\begin{lemma} \label{lemma: linear algebra 1 for embedding}
Let $V$ be a finite dimensional real vector space, $A$ a finite set, and $n$ a natural number.
Let $M$ be an $(n, n\dim V+1)$ matrix with entires in $A$ such that
no value appears twice in a row or in a column.
Then for almost every choice of $(u_a)_{a\in A} \in V^A$ the columns of the matrix
\[ (u_{M_{ij}})_{1\leq i \leq n, 1\leq j\leq n\dim V+1} \]
are affinely independent.
\end{lemma}
\begin{proof}
\textbf{Case 1}. Suppose $V=\mathbb{R}$.
We prove the statement by induction on $n$.
The statement is trivial for $n=1$. We consider the case $n\geq 2$.
Take $(t_a)_{a\in A}\in \mathbb{R}^A$.
Then the columns of the matrix $(t_{M_{ij}})_{1\leq i\leq n,1\leq j\leq n+1}$ are affinely independent
if and only if the following $(n+1,n+1)$ matrix is regular:
\begin{equation} \label{eq: matrix in linear algebra lemma 1}
  \begin{pmatrix}
   t_{M_{11}} & \dots & t_{M_{1 n+1}} \\
   \hdotsfor{3} \\
   t_{M_{n1}} & \dots & t_{M_{n n+1}} \\
   1 & \dots & 1
  \end{pmatrix}
\end{equation}
Set $\alpha=M_{11}$, and suppose that $\alpha$ appears $r$ times in $M$.
Then the determinant of the above matrix is a polynomial of $t_\alpha$ of degree $r$, and the coefficient of
the $t_\alpha^r$ term is equal to (up to $\pm$)
the determinant of the matrix of the same type (\ref{eq: matrix in linear algebra lemma 1})
of smaller size.
So by the induction the matrix (\ref{eq: matrix in linear algebra lemma 1}) is regular for almost every choice of
$(t_a)_{a\in A}\in \mathbb{R}^A$.

\textbf{Case 2}. Set $m=\dim V$, and take a basis $e_1,\dots,e_m$ of $V$.
Set $B=A\times \{1,2,\dots,m\}$ and
define an $(nm,nm+1)$ matrix $N$ valued in $B$ by
\[ N_{ij} = (M_{qj},r), \quad (i= (q-1)m + r, 1\leq q\leq n, 1\leq r\leq m)\]
for $1\leq i\leq nm$ and $1\leq j\leq nm+1$.
We can apply Case 1 to $N$; for almost every choice of $(t_b)_{b\in B}\in \mathbb{R}^B$ the columns of the matrix
$(t_{N_{ij}})$ are affinely independent.
Define $(u_a)_{a\in A}$ in $V^A$ by
\[ u_a = \sum_{r=1}^{m} t_{(a,r)} e_r.\]
Then the columns of $(u_{M_{ij}})_{1\leq i\leq n, 1\leq j\leq nm+1}$ are affinely independent for almost every choice
of $(t_b)$.
This proves the lemma.
\end{proof}

\begin{lemma} \label{lemma: linear algebra 2 for embedding}
Let $s$ and $n\leq N$ be positive integers.
For almost every choice of vectors $u_1,\dots,u_s\in V^{[N]}$,
if we choose at most $(n^k \dim V+1)$ distinct vectors in $V^{[n]}$ from
\[ u_a|_{b+[n]}\quad (1\leq a\leq s,\> b\in [N-n+1]) \]
then they are affinely independent.
Here $u_a|_{b+[n]}$ is the restriction of $u_a$ to $b+[n]\subset [N]$. We consider it as a vector in
$V^{[n]}$.
\end{lemma}
\begin{proof}
Set $m=\dim V$
and $A=\{1,2,\dots,s\}\times [N]$.
We can assume $s\geq n^km+1$.
Take distinct $(a_j,b_j)\in \{1,2,\dots,s\}\times [N-n+1]$
for $1\leq j\leq n^k m + 1$.
We define an $(n^k, n^k m+1)$ matrix $M$ valued in $A$ by $M_{ij}=(a_j,b_j+i)$ for
$i\in [n]$ and $1\leq j\leq n^km+1$.
Now $M$ satisfies the condition of Lemma \ref{lemma: linear algebra 1 for embedding}.
Hence for almost every choice of $(u_c)_{c\in A}\in V^A$ the columns of the matrix
$(u_{(a_j,b_j+i)})_{i\in [n],1\leq j\leq n^km+1}$ are affinely independent.
We define vectors in $V^{[N]}$ by
\[ u_a = (u_{(a,b)})_{b\in [N]}  \quad (1\leq a\leq s). \]
Then $u_{a_j}|_{b_j+[n]}$ $(1\leq j\leq n^km+1)$ are affinely independent.
\end{proof}

The next lemma strengthens  Lemma \ref{lemma: embedding of simplicial complex}
\begin{lemma} \label{lemma: simplicial complex lemma for embedding}
Let $n\leq N$ be positive integers, and $P$ a simplicial complex satisfying
\[ \dim P < \frac{n^k \dim V}{2}.\]
Then almost every linear map $f:P\to V^{[N]}$ satisfies the following.
If $x,y\in P$ and $a,b\in [N-n+1]$ satisfy
\[ f(x)|_{a+[n]} = f(y)|_{b+[n]},\]
then $x=y$ and $a=b$.
\end{lemma}
\begin{proof}
Let $v_1,\dots,v_s$ be the vertices of $P$.
Take vectors $u_1,\dots,u_s\in V^{[N]}$ and define a linear map $f:P\to V^{[N]}$ by
$f(v_i)= u_i$.
Let $x = \sum_{i=1}^s \alpha_{i}v_i$ and $y = \sum_{i=1}^s \beta_{i}v_i$
with $\sum_i \alpha_{i}=\sum_i \beta_{i}=1$.
Here $\alpha_i$ and $\beta_i$ are zero except for at most $\dim P+1$ ones respectively.
The equation $f(x)|_{a+[n]}=f(y)|_{b+[n]}$ implies
\[ \sum_{i=1}^s\alpha_iu_i|_{a+[n]} = \sum_{i=1}^s\beta_iu_i|_{b+[n]}.\]
The number of non-zero $\alpha_i$ and $\beta_i$ are at most $2(\dim P+1)\leq n^k\dim V+1$.
Now Lemma \ref{lemma: linear algebra 2 for embedding} implies the conclusion.
\end{proof}

\subsection{Idea of the proof} \label{subsection: setting and rough idea}

Throughout the rest of this section, we fix a positive number $\delta$ and a
continuous map $f=(f_1,f_2):X\to K^2$ (recall $K=[0,1]^D$).
We fix $0<\varepsilon<\delta$ so that for both $i=1,2$
\begin{equation} \label{eq: epsilon-delta for embedding}
  d(x,y) < \varepsilon \Longrightarrow |f_i(x)-f_i(y)|<\delta.
\end{equation}
Since $\mdim(X)<D/2^{k+1}$, we can find $\eta>0$ satisfying
\[ \mdim(X) < D\left(\frac{1}{2^{k+1}} -\eta\right).\]

We choose a positive integer $N$ sufficiently large so that
\begin{equation} \label{eq: fix of N for embedding}
   n\geq N \Longrightarrow \widim_\varepsilon(X,d_{[n]}) < D\left(\frac{1}{2^{k+1}}-\eta\right)(n-1)^k .
\end{equation}
We also assume that $N$ is even; this is just for simplicity of the exposition.

We now use the tiling construction of Section \ref{section: adding one dimension}. Let $M$ be a positive integer.
Since $X$ is assumed to have the marker property, we can find an integer $L\geq M$ and a continuous function $\phi:X\to [0,1]$ so that
\begin{itemize}
 \item If $\phi(x)>0$ at some $x\in X$, then $\phi(T^nx) = 0$ for all non-zero $n\in \mathbb{Z}^k$
  with $|n|<M$.
 \item For any $x\in X$ there exists $n\in \mathbb{Z}^k$ satisfying $|n|<L$ and $\phi(T^n x)=1$.
\end{itemize}
For each $x\in X$ let
$\mathbb{R}^{k+1} = \bigcup_{n\in \mathbb{Z}^k}V(x,n)$ be the Voronoi decomposition associated with
the set  $\{(n,1/\phi(T^nx))|\, n\in \mathbb{Z}^k\}$.
We choose a real number $H\geq (L+\sqrt{k})^2$ and set $W(x,n) = \pi_{\R^k}\left(V(x,n)\cap (\mathbb{R}^k\times \{-H\})\right)$.
These form a tiling of $\mathbb{R}^k$.

We choose $s>1$ satisfying
\begin{equation} \label{eq: choice of s}
  \frac{1}{2^{k+1}}-\eta < \frac{1}{2(1+s^k)}.
\end{equation}
Choosing the above $M$ sufficiently large, we can assume
(by Lemma \ref{lemma: Voronoi tiling} (4)) that for any $x\in X$, $n\in \mathbb{Z}^k$ and $a\in \mathbb{R}^k$
\begin{equation}  \label{eq: choice of M w.r.t. s}
   (a,-sH)\in V(x,n) \Longrightarrow B_{3\sqrt{k}N}(a/s + (1-1/s)n) \subset W(x,n).
\end{equation}

Now we have completed the setup for the proof of Proposition \ref{prop: main proposition for embedding}.
Before going further, we explain its strategy here.
The map $f=(f_1,f_2):X\to K^2$ has two components $f_1$ and $f_2$.
We first construct a perturbation $g_1$ of $f_1$ and next a perturbation $g_2$ of $f_2$.
These functions $g_1$ and $g_2$ play different roles.
Take $x\in X$.
We try to \textit{encode} the tiling
$\mathbb{R}^k = \bigcup_{n\in \mathbb{Z}^k} W(x,n)$ by  the value of $I_{g_1}(x) = (g_1(T^nx))_{n\in \mathbb{Z}^k}$.
If all the (non-empty) tiles $W(x,n)$ are sufficiently large, then this encoding can be done almost without error.
Such a good situation could be taken for granted in Section \ref{section: proof of embedding and symbolic system}; but in the current context this is not so ---
some tiles will be small. We cannot encode such small tiles
by $I_{g_1}(x)$, and this is the main difficulty we need to overcome in this proof.
Lemma \ref{lemma: most tiles are large} implies that such bad tiles are asymptotically negligible.
But this is not enough here.

Our argument goes as follows.
We give up on encoding all the information of the tiling $\bigcup_{n\in \mathbb{Z}^k} W(x,n)$.
Instead we construct a \textbf{``pseudo-tiling''} of $\mathbb{R}^k$ from the value of $I_{g_1}(x)$:
\begin{equation*}
  \begin{CD}
   \text{tiling $(W(x,n))_{n\in \mathbb{Z}^k}$}\\
   @V{\text{encode}}VV\\
   I_{g_1}(x) \in K^{\mathbb{Z}^k}\\
@V{\text{decode}}VV\\
   \text{pseudo-tiling $\mathcal{W} =(\mathcal{W}_n)_{n\in \mathbb{Z}^k}$}.
  \end{CD}
\end{equation*}
The pseudo-tiling $\mathcal{W}$ consists of
$\ell^\infty$-functions $\mathcal{W}_n\in \ell^\infty(\mathbb{Z}^k)$.
When $W(x,n)$ is sufficiently large, the function $\mathcal{W}_n$ is approximately equal to
the characteristic function of $W(x,n)\cap \mathbb{Z}^k$.
The definition of $\mathcal{W}$ will be given later.

Next we construct a perturbation $g_2(x)$ of $f_2(x)$ by using the pseudo-tiling associated with $x$.
Then the map $g= (g_1,g_2):X\to K^2$ has been constructed.
Take two points $x$ and $y$ in $X$, and suppose $(I_{g_1}(x),I_{g_2}(x))= (I_{g_1}(y), I_{g_2}(y))$.
The first equation $I_{g_1}(x) = I_{g_1}(y)$ implies that the pseudo-tilings associated with $x$ and $y$ are equal.
So $I_{g_2}(x)$ and $I_{g_2}(y)$ are constructed from the same pseudo-tiling.
Using this additional information, we can conclude $d(x,y) < \varepsilon$ from the equation $I_{g_2}(x)=I_{g_2}(y)$.
This last step is analogous to the situation of Proposition
\ref{prop: main proposition for embedding of zero dimensional extension}:
in that case we also had two equalities $I_g(x)=I_g(y)$ and $\pi(x)=\pi(y)$.
The equation $\pi(x)=\pi(y)$ implied that the Voronoi tilings associated with $x$ and $y$ are equal, and then
we deduced $d(x,y)<\varepsilon$ from $I_g(x)=I_g(y)$.

\subsection{Construction of \texorpdfstring{$g_1$}{g1} and encoding the tiling} \label{subsection: encoding}

In this subsection we construct a perturbation $g_1$ of $f_1$.
We choose
an $\varepsilon$-embedding $\pi:(X,d_{[N]})\to P$ where $P$ is a simplicial complex of
dimension $\widim_\varepsilon(X,d_{[N]})$.
Replacing $P$ with a sufficiently fine subdivision, we can apply Lemma \ref{lemma: approximation by linear map}
to $P$ and find a linear map $\tilde{f}_1:P \to K^{[N]}$ satisfying
\[ \norm{\tilde{f}_1(\pi(x))-I_{f_1}(x)|_{[N]}}_\infty
   \stackrel{\mathrm{def}}{=} \max_{n\in [N]}|\tilde{f}_1(\pi(x))_n-f_1(T^nx)| < \delta
   \quad (\forall x\in X).\]
Set $P' = P\times \{n\in \mathbb{Z}^k|\, |n|< L+\sqrt{k}N\}$, and consider the map
\begin{equation} \label{eq: preliminary map of F}
   P' \to K^{[N]}, \quad (x,n) \mapsto \tilde{f}_1(x).
\end{equation}
Note $\dim P' = \dim P = \widim_\varepsilon (X,d_{[N]}) < D N^k/2^{k+1}$.
We perturb the map (\ref{eq: preliminary map of F})
using Lemma \ref{lemma: simplicial complex lemma for embedding} (applied with $V=\mathbb{R}^D$ and $n=N/2$), and
construct a linear map $F:P'\to K^{[N]}$ satisfying the following two conditions\footnote{Recall that $N$ is even.}. Note that the map $F$ has two variables $x\in P$ and $n\in \mathbb{Z}^k$;
the second variable $n$ of $F$ will be used for encoding the positions of the Voronoi centers.

\begin{enumerate}[label=(\theequation)]
\refstepcounter{equation}
\item\label{condition: F (1)}
For all $x\in X$ and $n\in \mathbb{Z}^k$ with $|n|<L+\sqrt{k}N$
\[
\norm{F(\pi(x),n)-I_{f_1}(x)|_{[N]}}_{\infty} < \delta.\\
\]
\refstepcounter{equation}
\item\label{condition: F (2)}
 If $x,y\in P'$ and $a,b\in [N/2+1]$ satisfy
\[
F(x)|_{a+[N/2]} = F(y)|_{b+[N/2]},
\]
then $x=y$ and $a=b$.
\end{enumerate}

Take $x\in X$. We define $g_1(x)\in K$ as follows.
We take a cut-off function $\alpha:[0,\infty)\to [0,1]$ satisfying
  $\alpha(0) =0$ and $\alpha(t) = 1$ for $t\geq 1$.
If $0\in \partial W(x,n)$
 for some $n\in \mathbb{Z}^k$, then we set $g_1(x) = f_1(x)$.
Otherwise there uniquely exists $n\in \mathbb{Z}^k$ satisfying
$0\in \mathrm{int}W(x,n)$.
Taking $a\in \mathbb{Z}^k$ with $a\equiv n \pmod N$ and $0\in a+[N]$,
we set
\begin{multline*}
 g_1(x) = \left\{1-\alpha(\dist(0,\partial W(x,n)))\right\}f_1(x) \\
   + \alpha(\dist(0,\partial W(x,n)))F(\pi(T^ax),n-a)_{-a} .
\end{multline*}
Note that $0\in \mathrm{Int}\,W(x,n)$ implies $|n|<L+\sqrt{k}$ by Lemma \ref{lemma: Voronoi tiling} (3).
Hence $|n-a| < L + \sqrt{k}N$ and
 the term $F(\pi(T^ax),n-a)_{-a}$ is well-defined.
The map $g_1:X\to K$ is continuous since $W(x,n)$ depends continuously on $x$.
From Condition~\ref{condition: F (1)}, the function~$g_1$ satisfies $|g_1(x)-f_1(x)| < \delta$ because
\[ |F(\pi(T^ax),n-a)_{-a}-f_1(x)| = |F(\pi(T^ax),n-a)_{-a}-I_{f_1}(T^ax)_{-a}| < \delta.\]
We will need the following formula of $I_{g_1}(x)$ later.
The proof is the same as Claim \ref{claim: formula of I_g for small entropy factor}
in Section \ref{section: small entropy factors}.
\begin{claim} \label{claim: formula of I_g_1 for embedding}
Let $a,n\in \mathbb{Z}^k$ such that $a\equiv n \pmod N$ and
$a+[N]\subset \mathrm{int}_1 W(x,n)$. Then
\[  I_{g_1}(x)|_{a+[N]} = F(\pi(T^a x),n-a).\]
\end{claim}

\subsection{Decoding and the construction of pseudo-tiling}

For $n\in \mathbb{Z}^k$ with $|n|<L+\sqrt{k}N$ we define $Q_n\subset K^{[N]}$
as the set of $F(x,n)$ $(x\in P)$.
These $Q_n$ can be thought as the \textit{decoder} of the encoding $I_{g_1}$.
From Condition~\ref{condition: F (2)}, for $m,n\in \mathbb{Z}^k$ with $|m|,|n|<L+\sqrt{k}N$
and $a,b\in [N/2+1]$ we have
\[ Q_m|_{a+[N/2]}\cap Q_n|_{b+[N/2]}=\emptyset \quad \text{if $(m,a)\neq (n,b)$}.\]
We choose $\tau>0$ so that
\begin{equation} \label{eq: condition on tau}
    \tau < \min_{(m,a)\neq(n,b)}\dist(Q_m|_{a+[N/2]},Q_n|_{b+[N/2]})
\end{equation}
where the minimum is taken over all pairs of distinct $(m,a)$ and $(n,b)$ in $\mathbb{Z}^k\times [N/2+1]$ with $|m|,|n|<L+\sqrt{k}N$, and $\dist(\cdot,\cdot)$ is the Euclidean distance.

Let $1_{[N]} \in \ell^\infty(\mathbb{Z}^k)$ be the characteristic function of $[N]=\{0,1,\dots,N-1\}^k$.
We choose a cut-off function $\beta:[0,\infty)\to [0,1]$ such that $\beta(0)=1$ and
$\beta(t)=0$ for $t\geq \tau$.
Take a point $\omega\in K^{\mathbb{Z}^k}$.
For $n\in \mathbb{Z}^k$
we define an $\ell^\infty$-function $\mathcal{W}_n:\mathbb{Z}^k \to [0,1]$ by
\[ \mathcal{W}_n(t) = \min\left(1, \sum_{|a-n|<L+\sqrt{k}N}\beta(\dist(\omega|_{a+[N]}, Q_{n-a}))1_{[N]}(t-a)\right).\]

The function $\mathcal{W}_n$ is supported in $B_{L+\sqrt{k}N}(n+[N])$.
We set $\mathcal{W}(\omega) = (\mathcal{W}_n)_{n\in \mathbb{Z}^k}$.
This is the \textbf{pseudo-tiling} associated with $\omega$.
We sometimes denote $\mathcal{W}_n$ by $\mathcal{W}^{\omega}_n$.
Now the map
\[ K^{\mathbb{Z}^k}\ni \omega\longmapsto \mathcal{W}(\omega) \in (\ell^\infty(\mathbb{Z}^k))^{\mathbb{Z}^k}\]
is continuous.
It is also equivariant in the following sense:
Let $m\in \mathbb{Z}^k$, and take $\omega'\in K^{\mathbb{Z}^k}$ with $\omega'_n = \omega_{n+m}$.
Then $\mathcal{W}^{\omega'}_n(t) = \mathcal{W}^\omega_{n+m}(t+m)$.
The meaning of the terminology ``pseudo-tiling'' is clarified by the next lemma.
\begin{lemma} \label{lemma: pseudo-tiling}
Let $x\in X$ and set $(\mathcal{W}_n)_{n\in \mathbb{Z}^k}= \mathcal{W}(I_{g_1}(x))$ .
\begin{enumerate}
\item
Let $n\in \mathbb{Z}^k$.
If $m\in \mathbb{Z}^k$ satisfies $B_{2\sqrt{k}N}(m)  \subset W(x,n)$
then
\[ \mathcal{W}_n(m) =1 \qquad\text{and}\qquad \mathcal{W}_{n'}(m) = 0 \>(\forall n'\neq n).\]

\item
There exists $n\in \mathbb{Z}^k$ such that $\mathcal{W}_n(m)=1$ and $\mathcal{W}_{n'}(m)=0$
for every $m\in B_{\sqrt{k}N}((1-1/s)n)\cap \mathbb{Z}^k$, $n'\neq n$.
Here $s>1$ is the positive constant introduced in Subsection \ref{subsection: setting and rough idea}.
\end{enumerate}
\end{lemma}

\noindent
Roughly speaking, property~(1) of the lemma says that the function $\mathcal{W}_n$ looks like the characteristic function of $W(x,n)\cap \mathbb{Z}^k$
for ``nice'' tiles $W(x,n)$.

\begin{proof}
(1)
Take $a\in \mathbb{Z}^k$ satisfying $a\equiv n\pmod N$ and $m\in a+[N]$.
The $1$-neighborhood of $a+[N]$ is contained in $W(x,n)$.
Hence Claim~\ref{claim: formula of I_g_1 for embedding} implies
\[ I_{g_1}(x)|_{a+[N]} = F(\pi(T^ax),n-a)\in Q_{n-a},\]
which implies that $\mathcal{W}_n(m)=1$.

Next we prove $\mathcal{W}_{n'}(m)=0$ for $n'\neq n$ by showing
\[
\dist(I_{g_1}(x)|_{b+[N]},Q_{n'-b})\geq \tau
\]
for all $b\in \mathbb{Z}^k$ satisfying
$m\in b+[N]$ and $|n'-b|<L+\sqrt{k}N$.
We choose $c \in \mathbb{Z}^k$ satisfying $c\equiv n\pmod N$ and $b\in c+[N]$.
We define $b',c'\in \mathbb{Z}^k$ by
\begin{align*}
   b'_i &= b_i,\quad c'_i = c_i & \text{if $b_i-c_i< N/2$},\\
   b'_i &= c'_i = c_i+N & \text{if $b_i-c_i\geq N/2$}.
\end{align*}
Now $b'-c'\in [N/2]$, $b'-b\in [N/2+1]$,
$c'\equiv n\pmod N$ and
the $1$-neighborhood of $c'+[N]$ is contained in $B_{2\sqrt{k}N}(m)\subset W(x,n)$.
So $I_{g_1}(x)|_{c'+[N]} = F(\pi(T^{c'}x),n-c')$
by Claim \ref{claim: formula of I_g_1 for embedding}. Then
\begin{equation*}
  \begin{split}
   (I_{g_1}(x)|_{b+[N]})|_{b'-b+[N/2]} &= I_{g_1}(x)|_{b'+[N/2]} \\
   &= (I_{g_1}(x)|_{c'+[N]})|_{b'-c'+[N/2]} \\
   &= F(\pi(T^{c'}x),n-c')|_{b'-c'+[N/2]}.
  \end{split}
\end{equation*}
This is contained in $Q_{n-c'}|_{b'-c'+[N/2]}$.
Suppose \[\dist(I_{g_1}(x)|_{b+[N]},Q_{n'-b})<\tau.\]
Then we have
$\dist(Q_{n-c'}|_{b'-c'+[N/2]},Q_{n'-b}|_{b'-b+[N/2]})<\tau$.
By condition~(\ref{eq: condition on tau}) on $\tau$ it follows that
$n-c'=n'-b$ and $b'-c'=b'-b$, hence $c'=b$ and $n'=n$.

(2) Take $n\in \mathbb{Z}^k$ with $(0,-sH)\in V(x,n)$.
From condition~(\ref{eq: choice of M w.r.t. s}),
\[ B_{3\sqrt{k}N}((1-1/s)n)\subset W(x,n). \]
Then for $m\in B_{\sqrt{k}N}((1-1/s)n)\cap \mathbb{Z}^k$ we have
\[B_{2\sqrt{k}N}(m)\subset
B_{3\sqrt{k}N}((1-1/s)n)\subset W(x,n).\]
By~(1) above we get
$\mathcal{W}_n(m)=1$ and $\mathcal{W}_{n'}(m)=0$ for $n'\neq n$.
\end{proof}

\subsection{Construction of \texorpdfstring{$g_2$}{g2} and the proof of Proposition \ref{prop: main proposition for embedding}}

In this subsection we construct a perturbation $g_2$ of $f_2$ and prove Proposition~\ref{prop: main proposition
for embedding}.
For $u=(u_1,\dots,u_k)\in \mathbb{R}^k$ we set
$\lceil u\rceil = (\lceil u_1\rceil,\dots,\lceil u_k\rceil)\in \mathbb{Z}^k$
($\lceil u_i\rceil$ is the smallest integer not smaller than $u_i$).
Set $N' = \lceil sN\rceil$.
For each $n\in \mathbb{Z}^k$ we consider the distance
$d_{[N]\cup (\lceil (1-s)n\rceil +[N'])}$ on $X$.
Although this looks complicated, its geometric meaning is clear:
Consider the projection from $\mathbb{R}^k\times \{-H\}$ to $\mathbb{R}^k\times \{-sH\}$ with respect to the
center $(n,0)$. Then $(\lceil (1-s)n\rceil +[N'])\times \{-sH\}$
is approximately equal to the image of $[N]\times \{-H\}$ under this projection.
So this is a version of our trick of going down to $\mathbb{R}^k\times \{-sH\}$ from $\mathbb{R}^k \times \{-H\}$.

Using (\ref{eq: fix of N for embedding}) and (\ref{eq: choice of s}), we have
\begin{equation*}
  \begin{split}
   \widim_\varepsilon(X,d_{[N]\cup(\lceil (1-s)n\rceil +[N'])}) &\leq
   \widim_\varepsilon(X,d_{[N]}) +\widim_\varepsilon(X,d_{[N']}) \\
   &\hspace{-0.6in}< D\left(\frac{1}{2^{k+1}}-\eta\right) N^k + D\left(\frac{1}{2^{k+1}}-\eta\right)(N'-1)^k \\
   &\hspace{-0.6in}\leq  D\left(\frac{1}{2^{k+1}}-\eta\right)(N^k + s^k N^k) < \frac{DN^k}{2},
  \end{split}
\end{equation*}
so there exists an $\varepsilon$-embedding
\[\pi_n:(X,d_{[N]\cup (\lceil (1-s)n\rceil +[N'])})\to R_n\]
where $R_n$ is a simplicial complex of dimension~$<D N^k/2$.
Let $R$ be the disjoint union of $R_n$ over $|n|<L+3\sqrt{k}N$.
By Lemmas \ref{lemma: approximation by linear map} and \ref{lemma: embedding of simplicial complex}
we can find a linear embedding $G:R\to K^{[N]}$ satisfying
\[ \norm{G(\pi_n(x))-I_{f_2}(x)|_{[N]}}_\infty < \delta \quad
   (x\in X, |n|<L+3\sqrt{k}N).\]

We define  a continuous map $g_2:X\to K$ as follows.
For a real number $t$ we set $\langle t\rangle = \max(0,\min(1,t))\in [0,1]$, and for
$u =(u_1,\dots,u_D)\in \mathbb{R}^D$ we set $\langle u\rangle = (\langle u_1\rangle,\dots,\langle u_D\rangle)\in [0,1]^D =K$.
For each $n\in \mathbb{Z}^k$ we take $a_n \in \mathbb{Z}^k$ satisfying $a_n\equiv n\pmod N$ and $0\in a_n+[N]$.
Let $x\in X$. Let $(\mathcal{W}_n)_{n\in \mathbb{Z}^k}=\mathcal{W}(I_{g_1}(x))$ be the pseudo-tiling
associated with $I_{g_1}(x)\in K^{\mathbb{Z}^k}$.
We define $A(x)$ as the set of $n\in \mathbb{Z}^k$ with $\mathcal{W}_n(0)>0$.
Since $\mathcal{W}_n$ is supported in $B_{L+\sqrt{k}N}(n+[N])$,
every $n\in A(x)$ satisfies $|n|< L+2\sqrt{k}N$.
We set
\[ g_2(x) = \left\langle f_2(x)+
\frac{\sum_{n\in A(x)}\mathcal{W}_n(0)\left(G(\pi_{n-a_n}(T^{a_n}x))_{-a_n}-f_2(x)\right)}
{\max\left(1,\sum_{n\in A(x)}\mathcal{W}_n(0)\right)} \right\rangle, \]
where the term $G(\pi_{n-a_n}(T^{a_n}x))$ is well-defined because
\[|n-a_n| < L+2\sqrt{k}N+\sqrt{k}N = L+ 3\sqrt{k}N.\]
This satisfies
$|g_2(x)-f_2(x)|<\delta$ because
\begin{equation*}
  \begin{split}
   |G(\pi_{n-a_n}(T^{a_n}x))_{-a_n}-f_2(x)| &= |G(\pi_{n-a_n}(T^{a_n}x))_{-a_n}-I_{f_2}(T^{a_n}x)_{-a_n}| \\
   &\leq \norm{G(\pi_{n-a_n}(T^{a_n}x))-I_{f_2}(T^{a_n}x)|_{[N]}}_\infty <\delta.
  \end{split}
\end{equation*}

\begin{claim} \label{claim: formula of I_g_2 for embedding}
Let $x\in X$ and $a,n\in \mathbb{Z}^k$ with $a\equiv n\pmod N$.
Suppose the pseudo-tiling $(\mathcal{W}_n)_{n\in \mathbb{Z}^k}=\mathcal{W}(I_{g_1}(x))$
satisfies $\mathcal{W}_n=1$ and $\mathcal{W}_{n'}=0$ $(\forall n'\neq n)$ over $a+[N]$. Then
\[ I_{g_2}(x)|_{a+[N]} = G(\pi_{n-a}(T^a x)).\]
\end{claim}
\begin{proof}
Take $b\in a+[N]$. We have $a-b\equiv n-b \pmod N$ and $0\in (a-b)+[N]$.
Hence $a_{n-b} = a-b$.
Let $\mathcal{W}(I_{g_1}(T^b x)) = (\mathcal{W}'_m)_{m\in \mathbb{Z}^k}$.
Then $\mathcal{W}'_m(t) = \mathcal{W}_{m+b}(t+b)$ and
\begin{equation*}
   \mathcal{W}'_m(0) = \begin{cases}
                        1 & \quad (m=n-b)\\
                        0 & \quad (m\neq n-b).
                       \end{cases}
\end{equation*}
Therefore $A(T^b x) = \{n-b\}$ and
\begin{align*}
  g_2(T^b x) &= f_2(T^b x) + G(\pi_{n-b-a_{n-b}}(T^{a_{n-b}}T^b x))_{-a_{n-b}} - f_2(T^b x)\\
&  = G(\pi_{n-a}(T^a x))_{-a+b}.
\end{align*}
Hence $I_{g_2}(x)|_{a+[N]} = G(\pi_{n-a}(T^a x))$.
\end{proof}

The proof of Proposition \ref{prop: main proposition for embedding}
is completed by the next lemma.

\begin{lemma}
The map $I_{(g_1,g_2)}:X\to (K^2)^{\mathbb{Z}^k}$ is a $\delta$-embedding with respect to the distance $d$.
\end{lemma}
\begin{proof}
Suppose $(I_{g_1}(x),I_{g_2}(x)) = (I_{g_1}(y),I_{g_2}(y))$ for some $x,y\in X$.
Set \[(\mathcal{W}_n)_{n\in \mathbb{Z}^k}=\mathcal{W}(I_{g_1}(x))\ (=\mathcal{W}(I_{g_1}(y)).\]
From part~(2) of Lemma~\ref{lemma: pseudo-tiling} there exists $n\in \mathbb{Z}^k$ such that
$\mathcal{W}_n=1$ and $\mathcal{W}_{n'}=0$ $(n'\neq n)$ over $B_{\sqrt{k}N}((1-1/s)n)\cap \mathbb{Z}^k$.
Take $a\in \mathbb{Z}^k$ satisfying $a\equiv n \pmod N$ and $(1-1/s)n\in a+[0,N)^k$.
Then $a+[N]\subset B_{\sqrt{k}N}((1-1/s)n)$, and hence
$\mathcal{W}_n=1$ and $\mathcal{W}_{n'}=0$ $(n'\neq n)$ over $a+[N]$.
By Claim \ref{claim: formula of I_g_2 for embedding}
\[ I_{g_2}(x)|_{a+[N]} = G(\pi_{n-a}(T^a x)) = I_{g_2}(y)|_{a+[N]} = G(\pi_{n-a}(T^a y)).\]
Since $G$ is an embedding, we get $\pi_{n-a}(T^a x) = \pi_{n-a}(T^a y)$.
The map~$\pi_{n-a}$ is an $\varepsilon$-embedding with respect to
$d_{[N]\cup (\lceil (1-s)(n-a)\rceil + [N'])}$. Hence
\[ d_{\lceil (1-s)(n-a)\rceil + [N']}(T^a x,T^a y) = d_{\lceil (1-s)(n-a)\rceil + a+ [N']}(x,y) < \varepsilon.\]
Setting $(1-1/s)n=a+t$ with $t\in [0,N)^k$,
\begin{equation*}
  \begin{split}
   \lceil (1-s)(n-a)\rceil + a &= \lceil (1-s)n+(s-1)a + a\rceil = \lceil (1-s)n +sa\rceil \\
   &= \lceil -s t\rceil   \quad (\text{by $(s-1)n = sa+st$}) \\
   &= -\lfloor st\rfloor \in -[N'] \quad (\text{by $N' = \lceil sN\rceil$}).
  \end{split}
\end{equation*}
Therefore the origin is contained in $\lceil (1-s)(n-a)\rceil + a+ [N']$.
Thus $d(x,y)<\varepsilon<\delta$.
\end{proof}

\begin{remark}
By a little more careful argument, we can prove the following slightly stronger (but slightly more cumbersome to state) variant of Theorem~\ref{thm: embedding theorem for systems having marker property}:
Suppose $D_1$ and $D_2$ are positive numbers. Let $(X,\mathbb{Z}^k,T)$ be a dynamical system
having the marker property and
\[ \mdim(X) < \min\left(\frac{D_1}{2^{k+1}}, \frac{D_2}{4}\right).\]
Then for a dense $G_\delta$ subset of $f\in C(X,[0,1]^{D_1+D_2})$ the map
$I_f:X\to ([0,1]^{D_1+D_2})^{\mathbb{Z}^k}$ is an embedding.
\end{remark}

\vspace{0.5cm}

\address{Yonatan Gutman \endgraf
 Institute of Mathematics, Polish Academy of Sciences,
ul. \'{S}niadeckich~8, 00-656 Warszawa, Poland}

\textit{E-mail address}: \texttt{y.gutman@impan.pl}

\vspace{0.5cm}

\address{ Elon Lindenstrauss \endgraf
Einstein Institute of Mathematics, Hebrew University, Jerusalem 91904, Israel}

\textit{E-mail address}: \texttt{elon@math.huji.ac.il}

\vspace{0.5cm}

\address{ Masaki Tsukamoto \endgraf
Department of Mathematics, Kyoto University, Kyoto 606-8502, Japan}

\textit{E-mail address}: \texttt{tukamoto@math.kyoto-u.ac.jp}

\end{document}